\documentclass[11pt,leqno]{article}

\usepackage[a4paper, right=20mm, left=20mm, top=20mm]{geometry}

\usepackage{amsthm}     
\usepackage[english]{babel} 
\usepackage[utf8]{inputenc} 
\usepackage{csquotes} 
\usepackage{hyperref}
\usepackage{amsmath}
\usepackage{amssymb}
\usepackage{mathtools}

\usepackage{tikz}
\usetikzlibrary{matrix,arrows,calc}

\numberwithin{equation}{section}

\theoremstyle{plain}
\newtheorem{Theorem}{Theorem}[section]
\newtheorem{Lemma}[Theorem]{Lemma}

\newtheorem{Corollary}[Theorem]{Corollary}
\newtheorem*{Corollary*}{Corollary}

\theoremstyle{definition}
\newtheorem{Definition}[Theorem]{Definition}
\newtheorem{Remark}[Theorem]{Remark}

\newcommand{\R}{\mathbb{R}}
\newcommand{\Q}{\mathbb{Q}}
\newcommand{\Z}{\mathbb{Z}}
\newcommand{\F}{\mathbb{F}}
\newcommand{\Zp}{{\Z_p}}
\newcommand{\Qp}{{\Q_p}}

\newcommand{\id}{\mathrm{id}}
\newcommand{\Cdot}{\raisebox{-.8ex}{\scalebox{2}{$\cdot$}}}
\newcommand{\Hom}{\mathrm{Hom}}
\newcommand{\Ext}{\mathrm{Ext}}
\newcommand{\Stab}{\mathrm{Stab}}
\newcommand{\Rep}{\mathrm{Rep}}
\newcommand{\cok}{\mathrm{cok}}
\newcommand{\pr}{\mathrm{pr}}
\newcommand{\JSF}{\mathrm{JSF}}
\newcommand{\ASF}{\mathrm{ASF}}
\newcommand{\sign}{\mathrm{sign}}

\DeclareMathOperator{\ch}{ch}
\DeclareMathOperator{\tch}{tch}

\providecommand{\abs}[1]{\lvert#1\rvert}

\title{Coxeter combinatorics for sum formulas\\in the representation theory of algebraic groups}
\author{Jonathan Gruber}
\date{\today}

\begin{document}
	
\maketitle

\begin{abstract}
	Let $G$ be a simple algebraic group over an algebraically closed field $\F$ of characteristic $p\geq h$, the Coxeter number of $G$. We observe an easy `recursion formula' for computing the Jantzen sum formula of a Weyl module with $p$-regular highest weight. We also discuss a `duality formula' that relates the Jantzen sum formula to Andersen's sum formula for tilting filtrations and we give two different representation theoretic explanations of the recursion formula. As a corollary, we also obtain an upper bound on the length of the Jantzen filtration of a Weyl module with $p$-regular highest weight in terms of the length of the Jantzen filtration of a Weyl module with highest weight in an adjacent alcove.
\end{abstract}

\section*{Introduction}
\addcontentsline{toc}{section}{Introduction}

One of the most important problems in the representation theory of reductive algebraic groups in positive characteristic $p$ is to find the composition multiplicities of Weyl modules. When the characteristic is very large then these multiplicities can be computed as the values at $1$ of certain Kazhdan-Lusztig polynomials\footnote{This was proved by a reduction from algebraic groups to quantum groups \cite{AndersenJantzenSoergel} and from quantum groups to affine Lie algebras \cite{KL12,KL34}. The case of affine Lie algebras was treated in \cite{KashiwaraTanisaki}. In the non-simply-laced case, additional results from \cite{Lusztig} and \cite{KashiwaraTanisaki2} were required.}, but the $p$-canonical basis that provides a replacement for the Kazhdan-Lusztig basis in arbitrary characteristic\footnote{To be more precise, the $p$-canonical basis determines the multiplicities in good filtrations of tilting modules \cite{SmithTreumannTheory}. These in turn determine the characters of simple modules (and hence the composition multiplicities of Weyl modules) by the results of \cite{Sobaje}.} quickly becomes prohibitively hard to compute in higher rank.
For a specific weight $\lambda$, another tool for computing partial information about composition multiplicities in the Weyl module $\Delta(\lambda)$ is the Jantzen sum formula. By lifting to the $p$-adic integers a non-zero homomorphism from $\Delta(\lambda)$ to the induced module $\nabla(\lambda)$, one can define the \emph{Jantzen filtration}, an exhaustive descending filtration
\[ \Delta(\lambda)\supseteq \Delta(\lambda)^1 \supseteq \Delta(\lambda)^2 \supseteq \cdots \]
such that $\Delta(\lambda) / \Delta(\lambda)^1 = L(\lambda)$ is the unique simple quotient of $\Delta(\lambda)$ and the sum of the characters of the submodules $\Delta(\lambda)^i$ for $i>0$ can be computed explicitly using the \emph{Jantzen sum formula}. But again, the computation of this formula becomes quite tedious in large rank because it requires the determination of the dominant conjugates under the (finite) Weyl group $W_\mathrm{fin}$ of $G$ of many non-dominant weights (with respect to the dot action).

In this paper, we give a `recursion formula' that can be used to compute the Jantzen sum formula for a Weyl module with $p$-regular highest weight from the Jantzen sum formula for a Weyl module with highest weight in an adjacent alcove. It is convenient to identify the character lattice of a fixed $p$-regular linkage class with the anti-spherical module $M_\mathrm{asph}$ over the integral group ring of the affine Weyl group $W_\mathrm{aff}$ of $G$. The characters of the Weyl modules in the linkage class form a basis of its character lattice which is canonically indexed by the set $W_\mathrm{aff}^+$ of elements $x\in W_\mathrm{aff}$ whose length is minimal in the coset $W_\mathrm{fin} x$. These characters correspond to the standard basis $\{ N_x \mid x\in W_\mathrm{aff}^+ \}$ of $M_\mathrm{asph}$ in the aforementioned identification. Similarly, the Jantzen sum formula for a Weyl module indexed by $x\in W_\mathrm{aff}^+$ corresponds to an element $\JSF_x\in M_\mathrm{asph}$. With this notation in place, we can now state our recursion formula; see Section \ref{sec:reflections} for the definition of the affine reflection $s_{\beta,m} \in W_\mathrm{aff}$.
\medskip

\noindent
\textbf{Recursion formula.} \textit{ 
	Let $x\in W_\mathrm{aff}^+$ and let $s\in W_\mathrm{aff}$ be a simple reflection such that $x<xs$ in the Bruhat order and $xs \in W_\mathrm{aff}^+$. Write $xsx^{-1} = s_{\beta,m}$ for some $m>0$ and a positive root $\beta$. Then
	\[ \JSF_{xs} = \nu_p( m \cdot p ) \cdot N_x + \JSF_x \cdot s . \] }

In \cite{AndersenFiltrationsTilting}, H.H. Andersen defined a filtration on the space of homomorphisms from a Weyl module to a tilting module. Its construction is similar to that of the Jantzen filtration and the sum of the dimensions of the subspaces appearing in the filtration is given by the \emph{Andersen sum formula}. In Section \ref{sec:rewriting}, we explain how the Andersen sum formula for a fixed Weyl module indexed by $x\in W_\mathrm{aff}^+$ can be viewed as an element $\ASF_x$ of the dual $M_\mathrm{asph}^*$ of the anti-spherical module. Writing
\[ \langle - \,, - \rangle \colon M_\mathrm{asph}^* \times M_\mathrm{asph} \longrightarrow \Z \]
for the evaluation pairing and $\{ N_x^* \mid x\in W_\mathrm{aff}^+ \}$ for the `dual basis' of the standard basis of $M_\mathrm{asph}$, we can relate the two sum formulas by a simple `duality formula'.
\medskip

\noindent
\textbf{Duality formula.} \textit{ For $x,y\in W_\mathrm{aff}^+$, we have
\[ \big\langle \ASF_x, N_y \big\rangle = \big\langle N_x^*, \JSF_y \big\rangle . \] }

The formula is new in this formulation, but it has already appeared implicitly in work of H.H. Andersen and U. Kulkarni \cite{AndersenKulkarni}, see Remark \ref{rem:singularduality}.

Both the recursion formula and the duality formula are purely combinatorial observations at first, so one might wonder if it is possible to find a representation theoretic explanation for them as well (or to categorify them, as one might say). We give two versions of such representation theoretic explanations in Section \ref{sec:interpretation}. The first one (Theorem \ref{thm:JSF}) relates the cokernels of the natural maps from Weyl modules to induced modules with highest weights in adjacent alcoves (over the p-adic integers); these cokernels in turn determine the respective Jantzen sum formulas. The second one is via a `torsion Euler characteristic' that had already been used by H.H. Andersen and U. Kulkarni to study the two sum formulas (see Lemma \ref{lem:recursionEulercharacteristic} and Remark \ref{rem:recursionEulercharacteristic}).
As a corollary of Theorem \ref{thm:JSF}, we also obtain an upper bound on the length of the Jantzen filtration of a Weyl module with $p$-regular highest weight in terms of the length of the Jantzen filtration of a Weyl module with highest weight in an adjacent alcove.

\begin{Corollary*}
	Let $\lambda$ be a dominant weight in the fundamental alcove and let $x\in W_\mathrm{aff}^+$. Further let $s\in W_\mathrm{aff}$ be a simple reflection such that $x<xs$ and $xs\in W_\mathrm{aff}^+$ and write $xsx^{-1} = s_{\beta,m}$ for some $m>0$ and a positive root $\beta$. Then
	\[ \max\big\{ i \geq 0 \mathrel{\big|} \Delta(xs\Cdot \lambda)^i \neq 0 \big\} \leq 2 \cdot \max\big\{ i \geq 0 \mathrel{\big|} \Delta(x\Cdot \lambda)^i \neq 0 \big\} + \nu_p(m \cdot p) . \]
\end{Corollary*}

We give an example of how the recursion formula can be applied in Section \ref{sec:anexample}.

\subsection*{Acknowledgements}
\addcontentsline{toc}{subsection}{Acknowledgements}

The author would like to thank Donna Testerman and Thorge Jensen for numerous discussions and comments on the subject of this paper. Thanks are also due to Henning Haahr Andersen, Geordie Williamson and an anonymous referee for helpful comments on the manuscript. This project was supported by the Swiss National Science Foundation under grant number FNS 200020\textunderscore175571.

\section*{Notation} \label{sec:notation}
\addcontentsline{toc}{section}{Notation}

Let $G$ be a simply connected simple algebraic group over an algebraically closed field $\F$ of positive characteristic $p$. Let $X$ be the weight lattice of $G$ (with respect to some fixed maximal torus) and denote by $\Phi\subseteq X$ the root system of $G$ with positive system $\Phi^+$ corresponding to a base $\Delta$ of $\Phi$. We write $X^+$ for the set of dominant weights, $\rho$ for the half sum of positive roots and $\alpha_\mathrm{h}$ for the highest short root in $\Phi$. Let $W_\mathrm{fin}$ be the (finite) Weyl group of $G$ and fix a $W_\mathrm{fin}$-invariant inner product $\langle - \,, - \rangle$ on $X_\R=X \otimes_\Z \R$. For $\alpha\in\Phi^+$, we write $\alpha^\vee=2\alpha/\langle\alpha,\alpha\rangle$ for the coroot of $\alpha$
and $s_\alpha \in W_\mathrm{fin}$ for the reflection with
$ s_\alpha(x)=x-\langle x , \alpha^\vee \rangle \cdot \alpha $
for $x\in X_\R$.
We consider the dot action of $W_\mathrm{fin}$ on $X_\R$, given by
\[ w \Cdot x \coloneqq w( x+\rho ) - \rho \]
for $w\in W_\mathrm{fin}$ and $x\in X_\R$.

We write $\Rep(G)$ for the category of finite-dimensional $G$-modules and for any finite-dimensional $G$-module $M$, we write $\ch M \in \Z[X]^{W_\mathrm{fin}}$ for the character of $M$ and $[M]$ for the class of $M$ in the Grothendieck group $[\Rep(G)]$ of $\Rep(G)$.
For $\lambda\in X^+$, we denote by $L(\lambda)$, $\Delta(\lambda)$, $\nabla(\lambda)$ and $T(\lambda)$ the simple module, the Weyl module, the induced module and the indecomposable tilting module of highest weight $\lambda$, respectively, and we write $\chi_\lambda = \ch \Delta(\lambda) = \ch \nabla(\lambda)$.
See Sections II.2.1--4, II.2.13 and II.E.4 in \cite{Jantzen} for the definitions of these modules and note that $\Delta(\lambda)$ and $\nabla(\lambda)$ are denoted by $V(\lambda)$ and $H^0(\lambda)$ there, respectively.
The characters $\chi_\lambda$, for $\lambda\in X^+$, form a basis of the $\Z$-module $\Z[X]^{W_\mathrm{fin}}$ and for a $G$-module $M$ with
$\ch M = \sum_\lambda a_\lambda \cdot \chi_\lambda$,
we define $[ \ch M : \chi_\lambda ]\coloneqq a_\lambda$.
For arbitrary $\lambda\in X$, there exists a unique $\lambda^\prime\in W_\mathrm{fin}\Cdot\lambda$ with $\lambda^\prime+\rho \in X^+$ and if $\lambda^\prime\in X^+$ then there is a unique $w\in W_\mathrm{fin}$ with $\lambda=w\Cdot \lambda^\prime$. For such $\lambda^\prime$ and $w$, we define
\[ \chi_\lambda = \begin{cases}
\det(w) \cdot \chi_{\lambda^\prime} & \text{if } \lambda^\prime\in X^+ , \\ 0 & \text{otherwise}
\end{cases} \]
and
\[ [ \ch M : \chi_\lambda ] = \begin{cases}
\det(w) \cdot [ \ch M : \chi_{\lambda^\prime} ] & \text{if } \lambda^\prime \in X^+ \\
0 & \text{otherwise} .
\end{cases} \]
For a $G$-module $M$ and an indecomposable $G$-module $N$, we write $[ M : N ]_\oplus$ for the multiplicity of $N$ in a Krull-Schmidt decomposition of $M$.
Let $\Z_p$ be the ring of $p$-adic integers and denote by $G_{\Z_p}$ the algebraic group scheme over $\Z_p$ corresponding to $G$.
For $\lambda\in X^+$, we can define a Weyl module $\Delta_{\Z_p}(\lambda)$ and an induced module $\nabla_{\Z_p}(\lambda)$ over $G_{\Z_p}$ such that both $\Delta_\Zp(\lambda)$ and $\nabla_\Zp(\lambda)$ are free over $\Zp$ and $\Delta_{\Z_p}(\lambda)\otimes \F \cong \Delta(\lambda)$ and $\nabla_{\Z_p}(\lambda)\otimes \F \cong \nabla(\lambda)$ as $G$-modules. By Section II.B.4 in \cite{Jantzen}, we have
\[ \Ext_{G_\Zp}^i\big( \Delta_\Zp(\lambda) , \nabla_\Zp(\mu) \big) \cong \begin{cases}
\Zp & \text{if } \lambda=\mu \text{ and }i=0, \\
0 & \text{otherwise} 
\end{cases} \]
for all $\lambda,\mu\in X^+$ and $i\geq 0$. For each $\lambda\in X^+$, we fix a generator $c_\lambda$ of $\Hom_{G_\Zp}\big( \Delta_\Zp(\lambda) , \nabla_\Zp(\mu) \big)$.

\section{Filtrations and sum formulas} \label{sec:filtrations}

For any homomorphism of $\Z_p$-modules $\varphi \colon M \to N$, we can define a descending filtration
\[ M = F^0(\varphi) \supseteq F^1(\varphi) \supseteq F^2(\varphi) \supseteq \cdots \]
with $F^i(\varphi) \coloneqq \varphi^{-1}(p^i \cdot N)$ for $i\geq 0$. For $\lambda\in X^+$ and $i>0$, we write $\Delta(\lambda)^i$ for the submodule that is spanned by the image of $F^i(c_\lambda)$ in $\Delta(\lambda) \cong \Delta_\Zp(\lambda) \otimes \F$ and call the descending filtration
\[ \Delta(\lambda) = \Delta(\lambda)^0 \supseteq \Delta(\lambda)^1 \supseteq \Delta(\lambda)^2 \supseteq \cdots \]
the \emph{Jantzen filtration}. As $c_\lambda$ is unique up to a unit in $\Z_p$, this is independent of the choice of $c_\lambda$. Partial information about the layers of this filtration can be obtained from the following formula (see Section II.8.19 in \cite{Jantzen}):

\begin{Theorem}[Jantzen sum formula] \label{thm:citeJSF}
	For $\lambda\in X^+$, we have
	\[ \JSF_\lambda \coloneqq \sum_{i>0} \ch \Delta(\lambda)^i = - \sum_{\alpha\in\Phi^+} ~ \sum_{0 < m < \langle \lambda + \rho , \alpha^\vee \rangle} \nu_p(m) \cdot \chi_{\lambda-m\alpha} \]
	and
	\[ \Delta(\lambda) / \Delta(\lambda)^1 \cong L(\lambda) . \]
\end{Theorem}

Now let $T$ be a tilting $G_\Zp$-module. For $\lambda\in X^+$, we define
\[ F_\lambda(T) \coloneqq \Hom_{G_\Zp}\big( \Delta_\Zp(\lambda) , T \big) \qquad \text{and} \qquad E_\lambda(T) \coloneqq \Hom_{G_\Zp}\big( T , \nabla_\Zp(\lambda) \big) . \]
Both $F_\lambda(T)$ and $E_\lambda(T)$ are free over $\Zp$, see Section II.E.23 in \cite{Jantzen}.
There is a canonical homomorphism of $\Z_p$-modules $d_\lambda \colon F_\lambda(T) \to E_\lambda(T)^*$ with
\[ e \circ f = d_\lambda(f)(e) \cdot c_\lambda \]
for $f\in F_\lambda(T)$ and $e\in E_\lambda(T)$ and we define $F_\lambda^i(T) \coloneqq F^i(d_\lambda)$. Now if $T$ is a tilting $G$-module then there is a tilting $G_\Zp$-module $\hat{T}$ (unique up to isomorphism) with $T \cong \hat{T}\otimes \F$ and we write $\bar{F}_\lambda^i(T)$ for the subspace that is spanned by the image of $F_\lambda^i(\hat{T})$ in $\Hom_G\big( \Delta(\lambda) , T \big)\cong F_\lambda(\hat{T})\otimes \F$. The descending filtration
\[ \Hom_G\big( \Delta(\lambda) , T \big) = \bar{F}_\lambda^0(T) \supseteq \bar{F}_\lambda^1(T) \supseteq \bar{F}_\lambda^1(T) \supseteq \cdots \]
is called the \emph{Andersen filtration}.
As for the Jantzen filtration, there is a sum formula that gives partial information about the layers of the filtration; see Theorem 5.2 in \cite{AndersenKulkarni}.

\begin{Theorem}[Andersen sum formula] \label{thm:citeASF}
	For $\lambda\in X^+$ and a tilting $G$-module $T$, we have
	\[ \ASF(\lambda,T) \coloneqq \sum_{i>0} \dim \bar{F}_\lambda^i(T) = - \sum_{\alpha\in\Phi^+} ~ \sum_{ m \notin I(\lambda,\alpha)} \nu_p(m) \cdot [ \ch T : \chi_{\lambda-m\alpha} ] , \]
	where $I(\lambda,\alpha)=\{ m\in\Z \mid 0 \leq m \leq \langle \lambda+\rho , \alpha^\vee \rangle \}$,
	and
	\[ \dim\big( \bar{F}_\lambda^0(T) / \bar{F}_\lambda^1(T) \big) = [ T : T(\lambda) ]_\oplus . \]
\end{Theorem}
Note that the inner sum on the right hand side of the Andersen sum formula is infinite, but the multiplicity $[ \ch T : \chi_{\lambda-m\alpha} ]$ is zero for all but finitely many values of $m$.

\section{Reflections in the affine Weyl group} \label{sec:reflections}

The representation theory of $G$ is governed to a large extent by the affine Weyl group $W_\mathrm{aff}$ of $G$ (to be defined below) and the alcove geometry associated with its $p$-dilated dot action. In order to formulate our main results, we need to introduce some more notation and establish certain properties of reflections in $W_\mathrm{aff}$. Most of the material in this section is presumably well-known to experts.

The \emph{affine Weyl group} of $G$ is the semidirect product $W_\mathrm{aff}=\Z\Phi \rtimes W_\mathrm{fin}$ and we write $\gamma\mapsto t_\gamma$ for the canonical embedding of $\Z\Phi$ into $W_\mathrm{aff}$.
For $\alpha\in\Phi^+$ and $m\in\Z$, we define an affine reflection $s_{\alpha,m} \in W_\mathrm{aff}$ by $s_{\alpha,m} \coloneqq t_{m\alpha}s_{\alpha}$. With $S=\{ s_\alpha \mid \alpha\in\Delta \} \cup \{ s_{\alpha_\mathrm{h},1} \}$, the pair $(W_\mathrm{aff},S)$ is a Coxeter system whose length function we denote by $\ell \colon W_\mathrm{aff} \to \Z_{\geq 0}$; see Section II.6.3 in \cite{Jantzen} and the references therein. We write $W_\mathrm{aff}^+$ for the set of elements $x\in W_\mathrm{aff}$ whose length is minimal in the coset $W_\mathrm{fin}x$.

We consider the $p$-dilated dot action of $W_\mathrm{aff}$ on $X_\R$, given by
\[ t_\gamma w \Cdot x \coloneqq w( x+\rho ) + p \cdot \gamma - \rho \]
for $\gamma\in\Z\Phi$, $w\in W_\mathrm{fin}$ and $x\in X_\R$. 
The fixed points of an affine reflection $s=s_{\alpha,m}$ (with respect to the $p$-dilated dot action) in $X_\R$ form an affine hyperplane
\[ H_s = H_{\alpha,m}^p = \{ x\in X_\R \mid \langle x + \rho , \alpha^\vee \rangle = mp \} \]
and the connected components of $X_\R \setminus \bigcup_{\alpha,m} H_{\alpha,m}^p$ are called the \emph{alcoves} in $X_\R$. We say that a hyperplane $H$ \emph{separates} two alcoves $C$ and $C^\prime$ if $C$ and $C^\prime$ are contained in different connected components of $X_\R\setminus H$ and we call two alcoves \emph{adjacent} if they are separated by a unique hyperplane. The alcove
\[ C_\mathrm{fund} = \{ x\in X_\R \mid 0 < \langle x+\rho , \alpha^\vee \rangle < p \text{ for all }\alpha\in\Phi^+ \} \]
is called the \emph{fundamental alcove}.
The closure $\overline{C}_\mathrm{fund}$ is a fundamental domain for the action of $W_\mathrm{aff}$ on $X_\R$ and we call a dominant weight $\lambda\in X^+$ \emph{$p$-regular} if $\lambda \in x\Cdot C_\mathrm{fund}$ for some $x\in W_\mathrm{aff}$. By the linkage principle (see Section II.6 in \cite{Jantzen}), the category $\Rep(G)$ of finite dimensional $G$-modules decomposes as
\[ \Rep(G) = \bigoplus_{\lambda\in \overline{C}_\mathrm{fund}\cap X} \Rep_\lambda(G) , \]
where $\Rep_\lambda(G)$ denotes the \emph{linkage class} of $\lambda\in \overline{C}_\mathrm{fund}\cap X$, the full subcategory of $\Rep(G)$ consisting of the $G$-modules $M$ such that all composition factors of $M$ are of the form $L(w\Cdot\lambda)$ for some $w\in W_\mathrm{aff}$.

Now let us write $R(W_\mathrm{aff})=\{ s_{\alpha,m} \mid \alpha\in\Phi^+ , m\in\Z \}$ for the set of reflections in $W_\mathrm{aff}$ and $R(W_\mathrm{fin})$ for the set of reflections in $W_\mathrm{fin}$. For $y\in W_\mathrm{aff}$, we will be interested in sets of the form
\[ R_L(y) \coloneqq \{ s \in R(W_\mathrm{aff}) \mid sy < y \} , \]
where $<$ denotes the Bruhat order.
By Corollaries 1.4.4 and 1.4.5 in \cite{BjoernerBrentiCoxeter}, we have $\abs{R_L(y)}=\ell(y)$ and if $y=s_1 \cdots s_\ell$ is a reduced expression then
\begin{equation} \label{eq:reflectionsreducedexpression}
R_L(y) = \{ s_1 s_2 \cdots s_i \cdots s_2 s_1 \mid 1 \leq i \leq \ell \} .
\end{equation}
Furthermore, the length $\ell(y)$ equals the number of reflection hyperplanes separating the alcoves $C_\mathrm{fund}$ and $y\Cdot C_\mathrm{fund}$ by Theorem 4.5 in \cite{HumphreysCoxeter}. In fact, J. Humphreys shows that the set of reflection hyperplanes separating the alcoves $C_\mathrm{fund}$ and $y\Cdot C_\mathrm{fund}$ equals $\{ s_1 s_2 \cdots s_{i-1} \Cdot H_{s_i} \mid 1 \leq i \leq \ell \}$. Combining these results, we obtain the following alternative description of $R_L(y)$:

\begin{Lemma} \label{lem:hyperplaneseparates}
	For $y\in W_\mathrm{aff}$, we have
	\[ R_L(y) = \{ s \in R(W_\mathrm{aff}) \mid H_s \text{ separates } C_\mathrm{fund} \text{ and } y \Cdot C_\mathrm{fund} \} . \]
\end{Lemma}
\begin{proof}
	Fix a reduced expression $y=s_1 \cdots s_\ell$. For $1 \leq i \leq \ell$, the hyperplane of fixed points of the reflection $s_1 s_2 \cdots s_i \cdots s_2 s_1$ is $s_1 s_2 \cdots s_{i-1}\Cdot H_{s_i}$ and the claim follows from the above discussion.
\end{proof}

In order to prove the recursion formula and the duality formula, we will need three more lemmas about reflections in $W_\mathrm{aff}$.

\begin{Lemma} \label{lem:Rset}
	Let $x\in W_\mathrm{aff}$ and $s\in S$ such that $x<xs$. Then
	$R_L(xs)=R_L(x) \sqcup \{ xsx^{-1} \} $.
\end{Lemma}
\begin{proof}
	Let $x=s_1 \cdots s_\ell$ be a reduced expression. The assumption implies that $xs=s_1 \cdots s_\ell s$ is also a reduced expression and the claim follows from \eqref{eq:reflectionsreducedexpression} as $xsx^{-1} = s_1 s_2 \cdots s_\ell s s_\ell \cdots s_2 s_1$.
\end{proof}

For the following lemma, recall that $W_\mathrm{aff}^+$ denotes the set of elements $x\in W_\mathrm{aff}$ whose length is minimal in the coset $W_\mathrm{fin}x$.

\begin{Lemma} \label{lem:AtMostOneReflectionWfinCoset}
	For $x,y\in W_\mathrm{aff}^+$, there exists at most one reflection $s \in R_L(y)$ such that $sy \in W_\mathrm{fin} x$.
\end{Lemma}
\begin{proof}
	Let $s,s^\prime \in R_L(y)$ such that $sy \in W_\mathrm{fin}x$ and $s^\prime y \in W_\mathrm{fin} x$. Then
	\[ ss^\prime = (syx^{-1}) (s^\prime y x^{-1})^{-1} \in W_\mathrm{fin} , \]
	but $s \notin W_\mathrm{fin}$ because $s y < y$ and $y$ has minimal length in the coset $W_\mathrm{fin} y$. Hence we can write $s=s_{\beta,m}$ for some $m\neq 0$ and $\beta\in\Phi^+$ and $s^\prime = s_{\beta^\prime,m^\prime}$ for some $m^\prime \in\Z$ and $\beta^\prime \in\Phi^+$. As $-\rho$ is fixed by all elements of $W_\mathrm{fin}$ and as $ss^\prime \in W_\mathrm{fin}$, we have
	\[ p \cdot m\beta - \rho = s\Cdot(-\rho) = s^\prime\Cdot(-\rho) = p \cdot m^\prime \beta^\prime - \rho \]
	and therefore $m\beta = m^\prime \beta^\prime$. As $m\neq 0$, this implies that $\beta = \beta^\prime$ and $m=m^\prime$, so $s=s^\prime$.
\end{proof}

\begin{Lemma} \label{lem:reflectionWfinconjugate}
	Let $x\in W_\mathrm{aff}$ and $w\in W_\mathrm{fin}$. Then
	\[ w \big( R_L(x) \setminus R(W_\mathrm{fin}) \big) w^{-1} = w R_L(x) w^{-1} \setminus R(W_\mathrm{fin}) = R_L(wx) \setminus R( W_\mathrm{fin} ) . \]
\end{Lemma}
\begin{proof}
	The first equality is clear since $R(W_\mathrm{fin})$ is stable under conjugation by elements of $W_\mathrm{fin}$. Now suppose that $s\in R_L(x)$ such that $wsw^{-1} \notin R(W_\mathrm{fin})$. By Lemma \ref{lem:hyperplaneseparates}, the hyperplane $H_s$ separates the alcoves $C_\mathrm{fund}$ and $x \Cdot C_\mathrm{fund}$  and it follows that $w\Cdot H_s = H_{wsw^{-1}}$ separates $w\Cdot C_\mathrm{fund}$ and $wx \Cdot C_\mathrm{fund}$. However, $H_{wsw^{-1}}$ does not separate $C_\mathrm{fund}$ and $w \Cdot C_\mathrm{fund}$. Indeed, the point $-\rho$ is contained in the intersection of the closures of the alcoves $C_\mathrm{fund}$ and $w\Cdot C_\mathrm{fund}$, so any hyperplane separating these alcoves must contain $-\rho$ and therefore correspond to a reflection in $W_\mathrm{fin}$. Hence $C_\mathrm{fund}$ and $w\Cdot C_\mathrm{fund}$ belong to the same connected component of $X_\R \setminus H_{wsw^{-1}}$ and $wx\Cdot C_\mathrm{fund}$ belongs to the other one. Using Lemma \ref{lem:hyperplaneseparates} again, we conclude that $wsw^{-1} \in R_L(wx) \setminus R(W_\mathrm{fin})$, so
	\[ w R_L(x) w^{-1} \setminus R(W_\mathrm{fin}) \subseteq R_L(wx) \setminus R( W_\mathrm{fin} ) . \]
	Conversely, if $s\in R_L(wx) \setminus R(W_\mathrm{fin})$ then
	\[ w^{-1}sw \in w^{-1} \big( R_L(wx) \setminus R(W_\mathrm{fin}) \big) w = w^{-1} R_L(wx) w \setminus R(W_\mathrm{fin}) \subseteq R_L(x) \setminus R(W_\mathrm{fin}) \]
	and it follows that $s \in w \big( R_L(x) \setminus R(W_\mathrm{fin}) \big) w^{-1}$.
	We conclude that
	\[ R_L(wx) \setminus R(W_\mathrm{fin}) \subseteq w \big( R_L(x) \setminus R(W_\mathrm{fin}) \big) w^{-1} = w R_L(x) w^{-1} \setminus R(W_\mathrm{fin}) , \]
	as required.
\end{proof}

Now let $h=\langle \rho , \alpha_\mathrm{h}^\vee \rangle + 1$ be the \emph{Coxeter number} of $G$.
We conclude this section with a lemma about $W_\mathrm{aff}^+$ that is well-known to experts.

\begin{Lemma}
	Suppose that $p \geq h$. Then
	\[ W_\mathrm{aff}^+=\{ x \in W_\mathrm{aff} \mid x \Cdot C_\mathrm{fund} \cap X^+ \neq \varnothing \} . \]
\end{Lemma}
\begin{proof}
	First note that $\alpha_\mathrm{h}^\vee$ is the highest root in the dual root system $\Phi^\vee$, hence $\langle \rho , \beta^\vee \rangle \leq \langle \rho , \alpha_\mathrm{h}^\vee \rangle <p$ for all $\beta\in \Phi^+$. This implies that $0 \in C_\mathrm{fund}$ and that $x\Cdot C_\mathrm{fund} \cap X \neq \varnothing$ for all $x\in W_\mathrm{aff}$.
	
	Now let $X_\R^+ = \{ \lambda \in X_\R \mid \langle \lambda , \beta^\vee \rangle > 0 \text{ for all } \beta\in\Phi^+ \}$ be the dominant Weyl chamber in $X_\R$. As the finite Weyl group $W_\mathrm{fin}$ acts simply transitively on the set of Weyl chambers, every coset $W_\mathrm{fin} x$ contains a unique element $y$ with $y\Cdot C_\mathrm{fund} \subseteq X_\R^+ - \rho$ and we claim that $y$ has minimal length among the elements of $W_\mathrm{fin} x$. Indeed, it is straightforward to see that an alcove $C$ satisfies $C \subseteq X_\R^+-\rho$ if and only if none of the reflection hyperplanes $H_{\beta,0}^p$ with $\beta\in\Phi^+$ separate $C_\mathrm{fund}$ and $C$, so Lemma \ref{lem:Rset} implies that $y$ is the unique element of $W_\mathrm{fin} x$ with $R_L(y) \cap R(W_\mathrm{fin})=\varnothing$. As $\ell(y) = \abs{R_L(y)}$, the claim follows from Lemma \ref{lem:reflectionWfinconjugate}.
	
	For $y\in W_\mathrm{aff}$ with $y\Cdot C_\mathrm{fund} \subseteq X_\R^+ - \rho$ and for $\lambda \in y \Cdot C_\mathrm{fund} \cap X$, we have $0 < \langle \lambda+\rho , \alpha^\vee \rangle = \langle \lambda , \alpha^\vee \rangle+1$ for all $\alpha\in\Delta$ and it follows that $\lambda \in X^+$. Conversely, if $y\Cdot C_\mathrm{fund}$ contains a dominant weight then $y\Cdot C_\mathrm{fund} \cap (X_\R^+-\rho) \neq \varnothing$ and therefore $y\Cdot C_\mathrm{fund} \subseteq X_\R -\rho$ by connectedness of $y\Cdot C_\mathrm{fund}$. We conclude that an alcove $x\Cdot C_\mathrm{fund}$ contains a dominant weight if and only if $x$ has minimal length among the elements of $W_\mathrm{fin} x$, as required.
\end{proof}

\section{Rewriting the sum formulas} \label{sec:rewriting}

Suppose that $p\geq h$, the Coxeter number of $G$. In this section, we rewrite the sum formulas of Jantzen and Andersen in a character- and dimension-free way by interpreting them as objects of the anti-spherical $W_\mathrm{aff}$-module and its dual, respectively.

Let $\lambda\in X^+$ be $p$-regular and let $x\in W_\mathrm{aff}^+$ and $\lambda_0 \in C_\mathrm{fund}$ such that $\lambda = x\Cdot \lambda_0$. Recall that
\[ \JSF_\lambda = - \sum_{\alpha\in\Phi^+} ~ \sum_{0 < m < \langle \lambda + \rho , \alpha^\vee \rangle} \nu_p(m) \cdot \chi_{\lambda-m\alpha} = - \sum_{\alpha\in\Phi^+} ~ \sum_{0 < mp < \langle \lambda + \rho , \alpha^\vee \rangle} \nu_p(mp) \cdot \chi_{\lambda-mp\alpha}  \]
and note that for $\alpha\in\Phi^+$ and $m\in\Z$, the condition that
$ 0 < mp < \langle \lambda + \rho , \alpha^\vee \rangle $
is satisfied if and only if the reflection hyperplane $H_{\alpha,m}^p$ separates the alcoves $C_\mathrm{fund}$ and $x\Cdot C_\mathrm{fund}$. Furthermore, we have
\[ \lambda-mp\alpha = t_{-m\alpha} \Cdot \lambda = s_\alpha t_{m\alpha} s_\alpha \Cdot \lambda = s_\alpha \Cdot ( s_{\alpha,m} \Cdot \lambda ) \]
and so $\chi_{\lambda-mp\alpha}= - \chi_{ s_{\alpha,m} \Cdot \lambda }$.
Combining these observations with Lemma \ref{lem:hyperplaneseparates}, we obtain
\begin{equation} \label{eq:JSFRset}
\JSF_{x\Cdot\lambda_0} = \sum_{s\in R_L(x)} \nu_p\big( m(s) \cdot p \big) \cdot \chi_{sx\Cdot\lambda_0} ,
\end{equation}
where for $s=s_{\alpha,m}$, we set $m(s)\coloneqq m$.

Next, we wish to rewrite the formula \eqref{eq:JSFRset} in terms of the Coxeter combinatorics associated with the affine Weyl group.
Let $M_\mathrm{asph}\coloneqq \sign \otimes_{\Z[W_\mathrm{fin}]} \Z[W_\mathrm{aff}]$ be the \emph{anti-spherical $W_\mathrm{aff}$-module}, where $\sign$ denotes the sign representation of $W_\mathrm{fin}$. For $x\in W_\mathrm{aff}$, we write $N_x \coloneqq 1\otimes x \in M_\mathrm{asph}$, so that the elements $N_y$ with $y\in W_\mathrm{aff}^+$ form a $\Z$-basis of $M_\mathrm{asph}$. Note that for $w\in W_\mathrm{fin}$ and $x,y\in W_\mathrm{aff}$, we have
\[ N_{wx}=\det(w) \cdot N_x \qquad \text{and} \qquad N_{xy} = N_x \cdot y . \] 
Now fix $\lambda_0 \in C_\mathrm{fund} \cap X$ and recall that the characters $\chi_{x\Cdot\lambda_0}$ with $x\in W_\mathrm{aff}^+$ form a basis of the $\Z$-submodule $\Z[X]^{W_\mathrm{fin}}_{\lambda_0}$ of $\Z[X]^{W_\mathrm{fin}}$ spanned by the characters of all $G$-modules in $\Rep_{\lambda_0}(G)$. There is an isomorphism of $\Z$-modules
\[ \psi_{\lambda_0} \colon \Z[X]^{W_\mathrm{fin}}_{\lambda_0} \longrightarrow M_\mathrm{asph} \]
with $\chi_{y\Cdot\lambda_0} \mapsto N_y$ for all $y\in W_\mathrm{aff}^+$. (This can be upgraded to an isomorphism of $\Z[W_\mathrm{aff}]$-modules by identifying $\Z[X]^{W_\mathrm{fin}}_{\lambda_0}$ with the Grothendieck group of $\Rep_{\lambda_0}(G)$ and letting $s+1$ act via a wall-crossing functor $\Theta_s$ for all $s\in S$.) For $x\in W_\mathrm{aff}^+$, we define
\[ \JSF_x \coloneqq \psi_{\lambda_0}\big( \JSF_{x\Cdot \lambda_0} \big) . \]

\begin{Lemma} \label{lem:chitomx}
	For all $x\in W_\mathrm{aff}$, we have $\psi_{\lambda_0}\big( \chi_{x\Cdot \lambda_0} \big)=N_x$.
\end{Lemma}
\begin{proof}
	For $x\in W_\mathrm{aff}^+$, this is just the definition of $\psi_{\lambda_0}$. For an arbitrary $x\in W_\mathrm{aff}$, we can write $x=wy$ with $w\in W_\mathrm{fin}$ and $y\in W_\mathrm{aff}^+$ and it follows that
	\[ \psi_{\lambda_0}\big( \chi_{x\Cdot \lambda_0} \big) = \psi_{\lambda_0}\big( \det(w) \cdot \chi_{y\Cdot \lambda_0} \big) = \det(w) \cdot N_y = N_x , \]
	as claimed.
\end{proof}

\begin{Corollary} \label{cor:JSFRsetasph}
	For $x\in W_\mathrm{aff}^+$, we have
	\[ \JSF_x = \sum_{ \substack{ s \in R(W_\mathrm{aff}) \\ sx<x } } \nu_p\big( m(s) \cdot p \big) \cdot N_{sx} . \]
\end{Corollary}
\begin{proof}
	This is immediate from equation \eqref{eq:JSFRset} and Lemma \ref{lem:chitomx}.
\end{proof}

\begin{Remark}
	Fix a reduced expression $x=s_1 \cdots s_\ell$ for $x\in W_\mathrm{aff}^+$ and set
	\[ t_i = s_1 s_2 \cdots s_i \cdots s_2 s_1 \qquad \text{and} \qquad x_i = t_i x = s_1 \cdots s_{i-1} s_{i+1} \cdots s_\ell \]
	for $1 \leq i \leq \ell$. Then $R_L(x)= \{ t_i \mid 1 \leq i \leq \ell \}$ by \eqref{eq:reflectionsreducedexpression} and
	\[ \JSF_x = \sum_{i=1}^\ell \nu_p\big( m(t_i) \cdot p \big) \cdot N_{x_i} \]
	by Corollary \ref{cor:JSFRsetasph}.
\end{Remark}

In order to rewrite the Andersen sum formula, we work in the dual $M_\mathrm{asph}^*=\Hom_\Z(M_\mathrm{asph},\Z)$ of the anti-spherical module $M_\mathrm{asph}$. We write
\[ \langle - \, , - \rangle \colon M_\mathrm{asph}^* \times M_\mathrm{asph} \longrightarrow \Z \]
for the natural evaluation pairing. For $x\in W_\mathrm{aff}^+$, we define an element $N_x^*\in M_\mathrm{asph}^*$ by $\langle N_x^* , N_y \rangle = \delta_{x,y}$ for all $y\in W_\mathrm{aff}^+$. Then every element $\vartheta\in M_\mathrm{asph}^*$ can be written as a formal infinite sum
$ \vartheta = \sum_x a_x \cdot N_x^* $
with $a_x\in\Z$ for all $x\in W_\mathrm{aff}^+$. Note that we do not require that all but finitely many of the $a_x$ are zero, the infinite sum is to be understood in the sense that $\langle \vartheta , N_x \rangle = a_x$ for all $x\in W_\mathrm{aff}^+$. For $y\in W_\mathrm{aff}$, there exist unique elements $w\in W_\mathrm{fin}$ and $x\in W_\mathrm{aff}^+$ with $y=wx$ and we set $N_y^* \coloneqq \det(w) \cdot N_x^*$.

Now let $\lambda\in X^+$ be $p$-regular and let $x\in W_\mathrm{aff}^+$ and $\lambda_0\in C_\mathrm{fund}$ such that $\lambda=x\Cdot\lambda_0$. Recall that
\begin{multline*}
\hspace{1.5cm} \ASF(\lambda,T) = - \sum_{\alpha\in\Phi^+} ~ \sum_{ m \notin I(\lambda,\alpha) } \nu_p(m) \cdot [ \ch T : \chi_{\lambda-m\alpha} ] \\ = - \sum_{\alpha\in\Phi^+} ~ \sum_{ mp \notin I(\lambda,\alpha) } \nu_p(mp) \cdot [ \ch T : \chi_{\lambda-mp\alpha} ] , \hspace{1.5cm}
\end{multline*}
where $I(\lambda,\alpha)=\{ m\in\Z \mid 0 \leq m \leq \langle \lambda+\rho , \alpha^\vee \rangle \}$. For $\alpha\in\Phi^+$ and $m\in\Z$, the condition that $mp\notin I(\lambda,\alpha)$ is satisfied if and only if $m\neq 0$ and the hyperplane $H_{\alpha,m}^p$ does not separate the alcoves $C_\mathrm{fund}$ and $x\Cdot C_\mathrm{fund}$. As before, we have
$ [ \ch T : \chi_{\lambda-mp\alpha} ] = - [ \ch T : \chi_{s_{\alpha,m} \Cdot \lambda} ] $
and using Lemma \ref{lem:hyperplaneseparates} we can rewrite the Andersen sum formula as
\begin{equation} \label{eq:ASFRset}
\ASF(x\Cdot\lambda_0,T) = \sum_{ \substack{ s \in R(W_\mathrm{aff}) \setminus R(W_\mathrm{fin}) \\ sx>x } } \nu_p\big( m(s) \cdot p \big) \cdot [ \ch T : \chi_{sx\Cdot \lambda_0} ] .
\end{equation}
Now the characters of the indecomposable tilting modules in $\Rep_{\lambda_0}(G)$ form a basis of $\Z[X]^{W_\mathrm{fin}}_{\lambda_0}$, so the Andersen sum formula defines a $\Z$-linear map
\[ \ASF_\lambda \colon \Z[X]^{W_\mathrm{fin}}_{\lambda_0} \longrightarrow \Z \]
with $\langle \ASF_\lambda , \ch T \rangle = \ASF(\lambda,T)$ for every tilting module $T$ in $\Rep_{\lambda_0}(G)$, where as before, we write
\[ \langle - \, , - \rangle \colon \big(\Z[X]^{W_\mathrm{fin}}_{\lambda_0}\big)^* \times \Z[X]^{W_\mathrm{fin}}_{\lambda_0} \longrightarrow \Z \]
for the evaluation pairing. The isomorphism $\psi_{\lambda_0}\colon \Z[X]^{W_\mathrm{fin}}_{\lambda_0} \to M_\mathrm{asph}$ induces an isomorphism
\[ \psi_{\lambda_0}^* \colon M_\mathrm{asph}^* \longrightarrow \big( \Z[X]^{W_\mathrm{fin}}_{\lambda_0} \big)^* \]
and we define $\ASF_x\in M_\mathrm{asph}^*$ by the equality
\[ \psi_{\lambda_0}^*\big( \ASF_x \big) = \ASF_{x\Cdot\lambda_0} . \]

\begin{Lemma} \label{lem:chitomxdual}
	For all $y\in W_\mathrm{aff}$, the $\Z$-linear map
	$ [ - : \chi_{y\Cdot\lambda_0} ] \colon \Z[X]^{W_\mathrm{fin}}_{\lambda_0} \to \Z $ with $ \chi \mapsto [\chi:\chi_{y\Cdot \lambda_0} ] $
	satisfies
	$ [ - : \chi_{y\Cdot\lambda_0} ] = \psi_{\lambda_0}^*(N_y^*) $.
\end{Lemma}
\begin{proof}
	Let $w\in W_\mathrm{fin}$ and $x\in W_\mathrm{aff}^+$ such that $y=wx$. For $z\in W_\mathrm{aff}^+$, we have
	\begin{multline*}
	\qquad [ \chi_{z\Cdot \lambda_0} : \chi_{y\Cdot \lambda_0} ] = \det(w) \cdot \delta_{x,z} = \det(w) \cdot \langle N_x^* , N_z \rangle = \langle N_y^* , N_z \rangle \\
	= \big\langle N_y^* , \psi_{\lambda_0}\big( \chi_{z\Cdot \lambda_0} \big) \big\rangle = \big\langle \psi_{\lambda_0}^*( N_y^* ) , \chi_{z\Cdot \lambda_0} \big\rangle \qquad
	\end{multline*}
	and the claim follows.
\end{proof}

\begin{Corollary} \label{cor:ASFRsetasph}
	For $x\in W_\mathrm{aff}^+$, we have
	\[ \ASF_x = \sum_{ \substack{ s \in R(W_\mathrm{aff}) \setminus R(W_\mathrm{fin}) \\ sx>x } } \nu_p\big( m(s) \cdot p \big) \cdot N_{sx}^* . \]
\end{Corollary}
\begin{proof}
	By Lemma \ref{lem:chitomxdual} and equation \eqref{eq:ASFRset}, we have
	\begin{multline*}
	\psi_{\lambda_0}^*\bigg( \sum_{ \substack{ s \in R(W_\mathrm{aff}) \setminus R(W_\mathrm{fin}) \\ sx>x } } \nu_p\big( m(s) \cdot p \big) \cdot N_{sx}^* \bigg) = \sum_{ \substack{ s \in R(W_\mathrm{aff}) \setminus R(W_\mathrm{fin}) \\ sx>x } } \nu_p\big( m(s) \cdot p \big) \cdot [ - : \chi_{sx\Cdot\lambda_0} ] \\ = \ASF_{x\Cdot\lambda_0} = \psi_{\lambda_0}^*\big( \ASF_x \big)
	\end{multline*}
	and the claim follows.
\end{proof}

\section{The recursion formula and the duality formula} \label{sec:formulas}

In this section, we prove the two formulas that were announced in the introduction. We start with the recursion formula. 

\begin{Theorem}[Recursion formula] \label{thm:recursion}
	Let $x\in W_\mathrm{aff}^+$ and $s\in S$ such that $x<xs\in W_\mathrm{aff}^+$. Then
	\[ \JSF_{xs} = \nu_p\big( m(xsx^{-1}) \cdot p \big) \cdot N_x + \JSF_x \cdot s . \]
\end{Theorem}
\begin{proof}
	Recall from Corollary \ref{cor:JSFRsetasph} that
	\[ \JSF_{x} = \sum_{t\in R_L(x)} \nu_p\big( m(t) \cdot p \big) \cdot N_{tx} \qquad \text{and} \qquad \JSF_{xs} = \sum_{t\in R_L(xs)} \nu_p\big( m(t) \cdot p \big) \cdot N_{txs} , \]
	where $R_L(xs)=R_L(x) \sqcup \{ xsx^{-1} \}$ by Lemma \ref{lem:hyperplaneseparates}. We conclude that
	\begin{align*}
	\JSF_{xs} & = \sum_{t\in R_L(xs)} \nu_p\big( m(t) \cdot p \big) \cdot N_{txs} \\
	& = \nu_p\big( m(xsx^{-1}) \cdot p \big) \cdot N_x + \sum_{t\in R_L(x)} \nu_p\big( m(t) \cdot p \big) \cdot N_{txs} \\
	& = \nu_p\big( m(xsx^{-1}) \cdot p \big) \cdot N_x + \sum_{t\in R_L(x)} \nu_p\big( m(t) \cdot p \big) \cdot N_{tx} \cdot s \\
	& = \nu_p\big( m(xsx^{-1}) \cdot p \big) \cdot N_x + \JSF_x \cdot s ,
	\end{align*}
	as claimed.
\end{proof}

Next we prove the duality formula, which shows that the Jantzen filtration and the Andersen filtration are closely related to each other, at least on a combinatorial level.

\begin{Theorem}[Duality formula]
	For all $x,y\in W_\mathrm{aff}^+$, we have
	\[ \big\langle \ASF_x , N_y \big\rangle = \big\langle N_x^* , \JSF_y \big\rangle . \]
\end{Theorem}
\begin{proof}
	By Corollaries \ref{cor:JSFRsetasph} and \ref{cor:ASFRsetasph}, we have
	\[ \JSF_y = \sum_{ \substack{ s \in R(W_\mathrm{aff}) \\ sy<y } } \nu_p\big( m(s) \cdot p \big) \cdot N_{sy} \qquad \text{and} \qquad \ASF_x = \sum_{ \substack{ s \in R(W_\mathrm{aff}) \setminus R(W_\mathrm{fin}) \\ sx>x } } \nu_p\big( m(s) \cdot p \big) \cdot N_{sx}^* . \]
	If we write $\JSF_y=\sum_z a_{y,z} \cdot N_z$ in terms of the basis $\{ N_z \mid z \in W_\mathrm{aff}^+ \}$ then a summand $\nu_p\big( m(s) \cdot p \big) \cdot N_{sy}$ with $s\in R_L(y)$ contributes $\det(w) \cdot \nu_p\big( m(s) \cdot p \big)$ to $a_{y,z}$ precisely when $sy=wz$ for some $w\in W_\mathrm{fin}$. Furthermore, for any $z\in W_\mathrm{aff}^+$, there exists at most one $s\in R_L(y)$ with $sy \in W_\mathrm{fin} z$ by Lemma \ref{lem:AtMostOneReflectionWfinCoset}. We conclude that
	\[ a_{y,z} = \begin{cases}
	\det(w) \cdot \nu_p\big( m(s) \cdot p \big) & \text{if } sy=wz \text{ for some } s\in R_L(y) \text{ and } w\in W_\mathrm{fin} , \\
	0 & \text{otherwise} .
	\end{cases} \]
	Analogously, we can write $\ASF_x = \sum_z b_{x,z} \cdot N_z^*$ with
	\[ b_{x,z} = \begin{cases}
	\det(w) \cdot \nu_p\big( m(s) \cdot p \big) & \text{if } x<sx=wz \text{ for some } s\in R(W_\mathrm{aff}) \setminus R(W_\mathrm{fin}) \text{ and } w\in W_\mathrm{fin} , \\
	0 & \text{otherwise} .
	\end{cases} \]
	If $sy = wx$ for some $s\in R_L(y)$ and $w\in W_\mathrm{fin}$ then $s\notin W_\mathrm{fin}$ because $sy<y$ and $y$ is minimal in the coset $W_\mathrm{fin}y$. By Lemma \ref{lem:reflectionWfinconjugate}, we have $w^{-1}sw \in R_L(w^{-1}y) \setminus R(W_\mathrm{fin})$ and it follows that \[ x = (w^{-1}sw) w^{-1} y < w^{-1} y = (w^{-1} s w) x . \]
	This implies that
	\[ a_{y,x} = \det(w) \cdot \nu_p \big( m(s) \cdot p \big) = \det(w^{-1}) \cdot \nu_p \big( m(w^{-1}sw) \cdot p \big) = b_{x,y} . \]
	If $x<sx=wy$ for some $s\in R(W_\mathrm{aff}) \setminus R(W_\mathrm{fin})$ and $w \in W_\mathrm{fin}$ then $s\in R_L(wy)$ and $w^{-1}sw\in R_L(y)$ by Lemma \ref{lem:reflectionWfinconjugate}. Furthermore, we have $w^{-1}x  =(w^{-1}sw) y$ and it follows that
	\[ b_{x,y} = \det(w) \cdot \nu_p\big( m(s) \cdot p \big) = \det(w^{-1}) \cdot \nu_p\big( m(w^{-1}sw) \cdot p \big) = a_{y,x} . \]
	We conclude that
	\[ \big\langle \ASF_x , N_y \big\rangle = b_{x,y} = a_{y,x} = \big\langle N_x^* , \JSF_y \big\rangle \]
	for all $x,y\in W_\mathrm{aff}^+$.
\end{proof}

\begin{Remark} \label{rem:singularduality}
	The duality formula
	\[ \langle \ASF_x , N_y \rangle = \langle N_x^* , \JSF_y \rangle \]
	can be generalized to non-regular weights, even when the characteristic $p$ is smaller than the Coxeter number of $G$. For $\lambda\in X^+$, we define $\JSF_\lambda \in \Z[X]^{W_\mathrm{fin}}$ and $\ASF_\lambda \in \big(\Z[X]^{W_\mathrm{fin}}\big)^*$ by
	\[ \JSF_\lambda = \sum_{i>0} \ch \Delta(\lambda)^i \qquad \text{and} \qquad \langle \ASF_\lambda , \ch T \rangle = \sum_{i>0} \dim \bar{F}_\lambda^i(T) \]
	for every tilting module $T$, where $\langle -\,,-\rangle \colon \big(\Z[X]^{W_\mathrm{fin}}\big)^* \times \Z[X]^{W_\mathrm{fin}} \to \Z$ denotes the evaluation pairing. Writing $\chi_\lambda^* \in \big(\Z[X]^{W_\mathrm{fin}}\big)^*$ for the `dual basis' with $\langle \chi_\lambda^* , \chi_\mu \rangle=\delta_{\lambda,\mu}$ for all $\mu\in X^+$, we have the following natural extension of the duality formula to arbitrary weights:
	\[ \langle \ASF_\lambda , \chi_\mu \rangle = \langle \chi_\lambda^* , \JSF_\mu \rangle \]
	This formula has already appeared (implicitly) in \cite{AndersenKulkarni}. (Compare the bijection constructed in the proof of Proposition 4.9 in \cite{AndersenKulkarni} with the sum formulas from Section \ref{sec:filtrations}.) In the same paper, H.H. Andersen and U. Kulkarni also give a representation theoretic explanation using a certain \emph{torsion Euler characteristic}, see also Subsection \ref{subsec:Eulercharacteristic} below.
\end{Remark}

\section{Two interpretations of the recursion formula} \label{sec:interpretation}

In this section, we give two representation theoretic explanations for the recursion formula
\[ \JSF_{xs}=\nu_p(xsx^{-1})\cdot N_x + \JSF_x \cdot s \]
from Theorem \ref{thm:recursion}. In both cases, we use wall-crossing functors to relate the sum formulas for Weyl modules with highest weights in adjacent alcoves.

Recall that we write $G_\Zp$ for the group scheme over $\Zp$ corresponding to $G$. As the Jantzen filtration of $\Delta(\lambda)$ is defined using the generator $c_\lambda$ of $\Hom_{G_\Zp}\big( \Delta_{\Z_p}(\lambda) , \nabla_{\Z_p}(\lambda) \big)$, we work with lifts of translation functors to the category $\Rep(G_{\Z_p})$ of finitely generated rational $G_{\Z_p}$-modules. This is possible by results of H.H. Andersen, as explained below.

\subsection{Translation functors over \texorpdfstring{$\Zp$}{Zp}}

By Proposition 5.2 and Theorem 5.4 in \cite{AndersenFiltrationsCohomology}, the category $\Rep(G_{\Z_p})$ admits a decomposition into linkage classes
\[ \Rep(G_{\Z_p}) = \bigoplus_{\lambda\in \overline{C}_\mathrm{fund}\cap X} \Rep_\lambda(G_{\Z_p}) , \]
where $\Rep_\lambda(G_{\Z_p})$ denotes the full subcategory of $G_{\Z_p}$-modules $M$ such that $M \otimes \F$ belongs to $\Rep_\lambda(G)$. Using the projection functors $\pr_\lambda \colon \Rep(G_{\Z_p}) \to \Rep_\lambda(G_{\Z_p})$, it is possible to define mutually right and left adjoint exact translation functors
\[ T_\lambda^\mu \colon \Rep_\lambda(G_{\Z_p}) \to \Rep_\mu(G_{\Z_p}) \qquad \text{and} \qquad T_\mu^\lambda \colon \Rep_\mu(G_{\Z_p}) \to \Rep_\lambda(G_{\Z_p}) \]
for $\lambda,\mu\in \overline{C}_\mathrm{fund} \cap X$ just as in Section II.7 in \cite{Jantzen}, see Lemma 5.5 in \cite{AndersenFiltrationsCohomology}. For $G_{\Z_p}$-modules $M$ and $N$ in $\Rep_\lambda(G_{\Z_p})$ and $\Rep_\mu(G_{\Z_p})$, respectively, the adjunction between $T_\lambda^\mu$ and $T_\mu^\lambda$ gives rise to natural isomorphisms
\[ \mathrm{adj}_1 \colon \Hom_{G_\Zp}\big( M , T_\mu^\lambda N \big) \longrightarrow \Hom_{G_\Zp}\big( T_\lambda^\mu M , N \big) \]
and
\[ \mathrm{adj}_2 \colon \Hom_{G_\Zp}\big( N , T_\lambda^\mu M \big) \longrightarrow \Hom_{G_\Zp}\big( T_\mu^\lambda N , M \big) \]
such that
\begin{align*}
&\mathrm{adj}_1(f\circ g) = \mathrm{adj}_1(f)\circ T_\lambda^\mu(g), \qquad &\mathrm{adj}_2(f\circ g) = \mathrm{adj}_2(f)\circ T_\mu^\lambda(g) , \\
&\mathrm{adj}_1(T_\mu^\lambda f\circ g) = f \circ \mathrm{adj}_1(g), \qquad &\mathrm{adj}_2(T_\lambda^\mu f \circ g) = f \circ \mathrm{adj}_2(g)
\end{align*}
for all morphisms $f$ and $g$ in suitable $\Hom$-spaces. Now fix $\lambda\in C_\mathrm{fund}\cap X$ and $\mu\in \overline{C}_\mathrm{fund} \cap X$ such that $\Stab_{W_\mathrm{aff}}(\mu)=\{1,s\}$ for some $s\in S$. For $x\in W_\mathrm{aff}^+$ with $x<xs$ and $xs\in W_\mathrm{aff}^+$, we have $x\Cdot\mu\in X^+$ and there are isomorphisms
\[ T_\lambda^\mu \Delta_{\Z_p}(x\Cdot\lambda) \cong \Delta_{\Z_p}(x\Cdot\mu) \cong T_\lambda^\mu \Delta_{\Z_p}(xs\Cdot\lambda)
\qquad \text{and} \qquad
T_\lambda^\mu \nabla_{\Z_p}(x\Cdot\lambda) \cong \nabla_{\Z_p}(x\Cdot\mu) \cong T_\lambda^\mu \nabla_{\Z_p}(xs\Cdot\lambda) . \]
Furthermore, there are short exact sequences
\[ 0 \longrightarrow \Delta_\Zp(xs\Cdot\lambda) \xrightarrow{~\,i\,~} T_\mu^\lambda \Delta_\Zp(x\Cdot\mu) \xrightarrow{~\,\pi\,~} \Delta_\Zp(x\Cdot\lambda) \longrightarrow 0 \]
and
\[ 0 \longrightarrow \nabla_\Zp(x\Cdot\lambda) \xrightarrow{~\,i^\prime~} T_\mu^\lambda \nabla_\Zp(x\Cdot\mu) \xrightarrow{~\,\pi^\prime~} \nabla_\Zp(xs\Cdot\lambda) \longrightarrow 0 \]
with
\[ i=\mathrm{adj}_1^{-1}(\id_{\Delta_\Zp(x\Cdot\mu)}) , \qquad i^\prime=\mathrm{adj}_1^{-1}(\id_{\nabla_\Zp(x\Cdot\mu)}) \]
and
\[ \pi=\mathrm{adj}_2(\id_{\Delta_\Zp(x\Cdot\mu)}) , \qquad \pi^\prime=\mathrm{adj}_2(\id_{\nabla_\Zp(x\Cdot\mu)}) , \]
see Section 2.4 in \cite{AndersenFiltrationsTilting}.
Let us further consider the morphisms
\begin{align*}
r = \mathrm{adj}_2\big( \id_{\Delta_\Zp(x\Cdot\mu)} \big) & \colon T_\mu^\lambda \Delta_\Zp(x\Cdot\mu) \longrightarrow \Delta_\Zp(xs\Cdot\lambda) \\
r^\prime = \mathrm{adj}_2\big( \id_{\nabla_\Zp(x\Cdot\mu)} \big) & \colon T_\mu^\lambda \nabla_\Zp(x\Cdot\mu) \longrightarrow \nabla_\Zp(x\Cdot\lambda) \\
s = \mathrm{adj}_1^{-1}\big( \id_{\Delta_\Zp(x\Cdot\mu)} \big) & \colon \Delta_\Zp(x\Cdot\lambda) \longrightarrow T_\mu^\lambda \Delta_\Zp(x\Cdot\mu) \\
s^\prime = \mathrm{adj}_1^{-1}\big( \id_{\nabla_\Zp(x\Cdot\mu)} \big) & \colon \nabla_\Zp(xs\Cdot\lambda) \longrightarrow T_\mu^\lambda \nabla_\Zp(x\Cdot\mu) .
\end{align*}
The following two lemmas give certain relations for the composition of these homomorphisms and for their composition with the homomorphisms $c_{x\Cdot\lambda}$ and $c_{xs\Cdot\lambda}$.

\begin{Lemma}[Andersen] \label{lem:Andersen1}
	Let $m=\nu_p\big( \dim \Delta(x\Cdot\mu) \big)$. Then we have
	\[ r\circ i = p^m \cdot \id_{\Delta_\Zp(xs\Cdot\lambda)} , \qquad r^\prime \circ i^\prime = p^m \cdot \id_{\nabla_\Zp(x\Cdot\lambda)} , \]
	\[ \pi \circ s = p^m \cdot \id_{\Delta_\Zp(x\Cdot\lambda)} , \qquad \pi^\prime \circ s^\prime = p^m \cdot \id_{\nabla_\Zp(xs\Cdot\lambda)} . \]
\end{Lemma}
\begin{proof}
	See Lemma 2.4 in \cite{AndersenFiltrationsTilting}.
\end{proof}

\begin{Lemma}[Andersen] \label{lem:Andersen2}
	Let $\beta=T_\mu^\lambda c_{x\Cdot\mu}$. Up to units in $\Zp$, we have $c_{x\Cdot\mu}=T_\lambda^\mu c_{x\Cdot \lambda}$ and there are commutative diagrams
	\begin{equation*}
	\begin{tikzpicture}[baseline={([yshift=-.5ex]current bounding box.center)}]
	\node (C1) at (0,0) {$T_\mu^\lambda \Delta_\Zp(x\Cdot\mu)$};
	\node (C2) at (3,0) {$T_\mu^\lambda \nabla_\Zp(x\Cdot\mu)$};
	\node (C3) at (0,-2) {$\Delta_\Zp(x\Cdot\lambda)$};
	\node (C4) at (3,-2) {$\nabla_\Zp(x\Cdot\lambda)$};
	
	\draw[->] (C1) -- node[above] {\small $\beta$}  (C2);
	\draw[->] (C1) -- node[left] {\small $\pi$} (C3);
	\draw[->] (C2) -- node[right] {\small $r^\prime$} (C4);
	\draw[->] (C3) -- node[below] {\small $c_{x\Cdot\lambda}$} (C4);
	\end{tikzpicture}
	\qquad \text{and} \qquad
	\begin{tikzpicture}[baseline={([yshift=-.5ex]current bounding box.center)}]
	\node (C1) at (0,0) {$T_\mu^\lambda \Delta_\Zp(x\Cdot\mu)$};
	\node (C2) at (3,0) {$T_\mu^\lambda \nabla_\Zp(x\Cdot\mu)$};
	\node (C3) at (0,-2) {$\Delta_\Zp(xs\Cdot\lambda)$};
	\node (C4) at (3,-2) {$\nabla_\Zp(xs\Cdot\lambda)$};
	
	\draw[->] (C1) -- node[above] {\small $\beta$}  (C2);
	\draw[->] (C3) -- node[left] {\small $i$} (C1);
	\draw[->] (C2) -- node[right] {\small $\pi^\prime$} (C4);
	\draw[->] (C3) -- node[below] {\small $c_{xs\Cdot\lambda}$} (C4);
	\end{tikzpicture}
	\end{equation*}
\end{Lemma}
\begin{proof}
	The short exact sequence
	\[ 0 \longrightarrow \nabla_\Zp(x\Cdot\lambda) \xrightarrow{~\,i^\prime~} T_\mu^\lambda \nabla_\Zp(x\Cdot\mu) \xrightarrow{~\,\pi^\prime~} \nabla_\Zp(xs\Cdot\lambda) \longrightarrow 0 \]
	gives rise to an exact sequence
	\begin{align*}
	0 \to \Hom_{G_\Zp}\big( \Delta_\Zp(x\Cdot\lambda) , \nabla_\Zp(x\Cdot\lambda) \big) \to \Hom_{G_\Zp}& \big( \Delta_\Zp(x\Cdot\lambda) , T_\mu^\lambda \nabla_\Zp(x\Cdot\mu) \big) \\ & \to \Hom_{G_\Zp}\big( \Delta_\Zp(x\Cdot\lambda) , \nabla_\Zp(xs\Cdot\lambda) \big) = 0 ,
	\end{align*}
	so the map $\varphi \mapsto \mathrm{adj}_1(i^\prime\circ \varphi)$ gives an isomorphism
	\begin{align*}
	\Hom_{G_\Zp}\big( \Delta_\Zp(x\Cdot\lambda) , \nabla_\Zp(x\Cdot\lambda) \big) & \cong \Hom_{G_\Zp}\big( \Delta_\Zp(x\Cdot\lambda) , T_\mu^\lambda \nabla_\Zp(x\Cdot\mu) \big) \\
	& \cong \Hom_{G_\Zp}\big( T_\lambda^\mu \Delta_\Zp(x\Cdot\lambda) , \nabla_\Zp(x\Cdot\mu) \big) \\
	& \cong \Hom_{G_\Zp}\big( \Delta_\Zp(x\Cdot\mu) , \nabla_\Zp(x\Cdot\mu) \big) .
	\end{align*}
	Hence up  to a unit in $\Z_p$, we have
	\[ c_{x\Cdot\mu} = \mathrm{adj}_1(i^\prime\circ c_{x\Cdot\lambda}) = \mathrm{adj}_1(i^\prime) \circ T_\lambda^\mu c_{x\Cdot\lambda} = \id_{\nabla_\Zp(\mu)} \circ T_\lambda^\mu c_{x\Cdot\lambda} = T_\lambda^\mu c_{x\Cdot\lambda} \]
	matching the first claim. Commutativity of the diagrams follows from Lemma 2.5 in \cite{AndersenFiltrationsTilting}.
\end{proof}

\begin{Remark}
	The first statement in Lemma \ref{lem:Andersen2} is true in greater generality, as has been pointed out by the referee:
	For $\delta,\nu\in C_\mathrm{fund}$ such that $x\Cdot\nu\in X^+$ and $x\Cdot\nu$ belongs to the \emph{upper closure} of the \emph{facet} containing $x\Cdot\delta$, we have $T_\delta^\nu c_{x\Cdot\delta} = c_{x\Cdot\nu}$.
	(See Section II.6.2 in \cite{Jantzen} for the terminology.)
	
	Indeed, by Proposition II.7.13 and Lemma II.B.9 in \cite{Jantzen}, the module $T_\nu^\delta \nabla_\Zp(x\Cdot\nu)$ has a filtration with subquotients the induced modules $\nabla_\Zp(xw\Cdot\delta)$ for $w\in \Stab_{W_\mathrm{aff}}(\nu)$, each weight $xw\Cdot\delta$ occurring precisely once. As $x\Cdot\delta$ is minimal among these weights, the first submodule in this filtration can be chosen to be isomorphic to $\nabla_\Zp(x\Cdot\delta)$ (see Remark 4 in Section II.4.16 in \cite{Jantzen}), so there is a short exact sequence
	\[ 0 \longrightarrow \nabla_\Zp(x\Cdot\delta) \longrightarrow T_\nu^\delta \nabla(x\Cdot\delta) \longrightarrow M \longrightarrow 0 . \]
	Now the claim follows as in the proof of Lemma \ref{lem:Andersen2}.
\end{Remark}

\subsection{Cokernels and exact sequences}

Recall from Section \ref{sec:filtrations} that every homomorphism of $\Z_p$-modules $\varphi \colon M \to N$ gives rise to a filtration
\[ M = F^0(\varphi) \supseteq F^1(\varphi) \supseteq F^2(\varphi) \supseteq \cdots \]
with $F^i(\varphi) \coloneqq \varphi^{-1}(p^i \cdot N)$ for $i\geq 0$. We write $\bar{F}^i(\varphi)$ for the subspace spanned by the image of $F^i(\varphi)$ in $M \otimes \F$. Under some additional hypotheses on $\varphi$, we can give an alternative description of the length of this filtration and of the sum of the dimensions of the vector spaces $\bar{F}^i(\varphi)$ for $i>0$.

\begin{Definition}
	Let $M$ be a finitely generated torsion $\Zp$-module. Then $M \cong \prod_{i=1}^r \Zp / p^{d_i} \Zp$ for certain positive integers $d_1,\ldots,d_r$ and we write
	\[ \ell(M) = \sum_{i=1}^r d_i \qquad \text{and} \qquad \mathrm{tmax}(M) = \max\{ d_i \mid 1 \leq i \leq r \} \]
	for the \emph{composition length} and the \emph{maximal torsion} of $M$, respectively.
\end{Definition}

Let us write $\Qp$ for the field of $p$-adic numbers. The following lemma is well-known, see for instance Section II.8.18 in \cite{Jantzen}.

\begin{Lemma} \label{lem:DimensionSumFiltrationLength}
	Let $M$ and $N$ be free $\Zp$-modules of finite rank and let $\varphi \colon M \to N$ be a homomorphism of $\Zp$-modules such that $\varphi \otimes \Qp$ is an isomorphism. Then $\cok(\varphi)$ is a torsion $\Zp$-module and
	\[ \sum_{i>0} \dim \bar{F}^i(\varphi) = \ell\big( \cok(\varphi) \big) \qquad \text{and} \qquad \max \{ i \mid \bar{F}^i(\varphi) \neq 0 \} = \mathrm{tmax}\big( \cok(\varphi) \big) . \]
\end{Lemma}
\begin{proof}
	By applying the exact functor $-\otimes\Qp$ to the exact sequence
	\[ 0 \to \ker(\varphi) \to M \xrightarrow{~\varphi~} N \to \cok(\varphi) \to 0 \]
	and using the fact that a $\Zp$-module $M^\prime$ is torsion if and only if $M^\prime\otimes\Qp \cong 0$, we see that $\ker(\varphi)$ and $\cok(\varphi)$ are torsion $\Zp$-modules.
	As $M$ is free, it follows that $\varphi$ is injective and by the elementary divisor theorem, we can choose $\Z_p$-bases $m_1,\ldots,m_r$ of $M$ and $n_1,\ldots,n_r$ of $N$ such that $\varphi(m_i)=a_i\cdot n_i$ for some $0 \neq a_i\in\Zp$ for all $i$. Then $F^j(\varphi)$ has a basis given by
	\[ \{ m_i \mid \nu_p(a_i)\geq j \} \cup \{ p^{j-\nu_p(a_i)} \cdot m_i \mid \nu_p(a_i)<j \} , \]
	so $\{ m_i \otimes 1 \mid \nu_p(a_i)\geq j \}$ is a basis of the vector space $\bar{F}^j(\varphi)$ and it follows that
	\[ \sum_{j>0} \dim \big( \bar{F}^j(\varphi) \big) = \sum_{j>0} \abs{ \{ i \mid \nu_p(a_i)\geq j \} } = \abs{ \{ (i,j) \mid \nu_p(a_i)\geq j \} } = \sum_{i=1}^{r} \nu_p(a_i) = \ell\big( \cok(\varphi) \big) . \]
	Furthermore, the description of the bases of the vector spaces $\bar{F}^j(\varphi)$ implies that
	\[ \max\{ j \geq 0 \mid \bar{F}^j(\varphi) \neq 0 \} = \max\{ \nu_p(a_i) \mid 1 \leq i \leq r \} = \mathrm{tmax}\big( \cok(\varphi) \big) , \]
	as claimed.
\end{proof}

Every rational $G_{\Z_p}$-module $M$ has a weight space decomposition $M=\bigoplus_{\delta\in X} M_\delta$. If $M$ is a torsion module then so are all of the weight spaces and this allows us to define a character for $M$.

\begin{Definition}
	Let $M$ be a torsion $G_\Zp$-module. The \emph{torsion character} of $M$ is
	\[ \tch M \coloneqq \sum_{\delta\in X} \ell(M_\delta) \cdot e^\delta \in \Z[X]^{W_\mathrm{fin}} . \]
\end{Definition}

\begin{Remark} \label{rem:torsioncharacter}
	Let us write $\Rep^\mathrm{tor}(G_\Zp)$ for the category of finitely generated torsion $G_\Zp$-modules and $[\Rep^\mathrm{tor}(G_\Zp)]$ for its Grothendieck group. It is straightforward to see from the definition that characters of torsion $G_{\Z_p}$-modules are additive on short exact sequences, so the torsion character defines an isomorphism between $[\Rep^\mathrm{tor}(G_\Zp)]$ and $\Z[X]^{W_\mathrm{fin}}$ that sends the isomorphism class of the torsion $G_\Zp$-module $\nabla_\Zp(\lambda) \otimes \big( \Zp / p \Zp \big)$ to $\chi_\lambda$ for all $\lambda\in X^+$.
	(Indeed, all irreducible $G_\Zp$-modules are torsion modules over $\Zp$ and remain irreducible under extension of scalars to $\F$, so the torsion character takes a basis of $[\Rep^\mathrm{tor}(G_\Zp)]$ to a basis of $\Z[X]^{W_\mathrm{fin}}$ because the corresponding statement is true for $[\Rep(G)]$.)
\end{Remark}

\begin{Lemma} \label{lem:filtrationsumcharacter}
	Let $\varphi \colon M \to N$ be a homomorphism of $G_{\Z_p}$-modules. Further suppose that $M$ and $N$ are free of finite rank over $\Z_p$ and that $\varphi \otimes \Qp$ is an isomorphism. Then
	\[ \sum_{i>0} \ch \bar{F}^i(\varphi) = \tch \cok(\varphi) . \]
\end{Lemma}
\begin{proof}
	By Lemma \ref{lem:DimensionSumFiltrationLength}, the $G_{\Z_p}$-module $\cok(\varphi)$ is torsion, so that $\tch \cok(\varphi)$ is well-defined.
	For $\delta\in X$, denote by $\varphi_\delta \colon M_\delta \to N_\delta$ the restriction of $\varphi$ to the $\delta$-weight spaces of $M$ and $N$. As $M$ and $N$ are free of finite rank, so are $M_\delta$ and $N_\delta$, and as $\varphi\otimes\Qp$ is an isomorphism, so is $\varphi_\delta \otimes \Qp$. The filtration $F^\bullet(\varphi)$ is compatible with the weight space decomposition (that is $F^i(\varphi_\delta)=F^i(\varphi)_\delta$) and Lemma \ref{lem:DimensionSumFiltrationLength} implies that
	\[ \sum_{i>0} \dim \bar{F}^i(\varphi)_\delta = \sum_{i>0} \dim \bar{F}^i(\varphi_\delta) = \ell\big(\cok(\varphi_\delta)\big) = \ell\big( \cok(\varphi)_\delta \big) . \]
	Now the claim is immediate from the definition of torsion characters.
\end{proof}

\begin{Corollary} \label{cor:JSFcharactercokernel}
	For $\delta\in X^+$, we have
	\[ \JSF_\delta = \tch\big( \cok (c_\delta) \big) . \]
\end{Corollary}
\begin{proof}
	By definition, we have $\JSF_\delta=\sum_{i>0} \ch \Delta(\delta)^i$, where $\Delta(\delta)^i=\bar{F}^i(c_\delta)$. Now $c_\delta \otimes \Qp$ is an isomorphism between the simple $G_\Qp$-modules $\Delta_\Qp(\delta)$ and $\nabla_\Qp(\delta)$ of highest weight $\delta$ and the claim follows from Lemma \ref{lem:filtrationsumcharacter}.
\end{proof}

By Remark \ref{rem:torsioncharacter} and Corollary \ref{cor:JSFcharactercokernel}, we can compute $\JSF_\delta$ by computing the class $[\cok(c_\delta)]$ in the Grothendieck group of finitely generated torsion $G_\Zp$-modules, for all $\delta\in X^+$.

As before, let us fix $\lambda\in C_\mathrm{fund}\cap X$ and $\mu\in \overline{C}_\mathrm{fund}\cap X$ with $\Stab_{W_\mathrm{aff}}(\mu)=\{1,s\}$. The \emph{$s$-wall crossing functor} on $\Rep_\lambda(G)$ is defined as $\Theta_s = T_\mu^\lambda \circ T_\lambda^\mu$. Writing $\Rep_\lambda^\mathrm{tor}(G_\Zp)$ for the category of finitely generated torsion $G_\Zp$-modules in $\Rep_\lambda(G_\Zp)$, there is an isomorphism of abelian groups
\[ [\Rep_\lambda^\mathrm{tor}(G_\Zp)] \longrightarrow M_\mathrm{asph} \]
with $\big[ \nabla_\Zp(x\Cdot \lambda) \otimes \big(\Zp / p \Zp \big) \big] \mapsto N_x$. The latter can be upgraded to an isomorphism of $\Z[W_\mathrm{aff}]$-modules by setting
$ [ M ] \cdot (s+1) \coloneqq [ \Theta_s M ] $
for $s\in S$ and for a $G_\Zp$-module $M$ in $\Rep_\lambda^\mathrm{tor}(G_\Zp)$. For the remainder of this section, we write $\Delta_x = \Delta_\Zp(x\Cdot \lambda)$, $\nabla_x = \nabla_\Zp(x\Cdot \lambda)$ and $c_x=c_{x\Cdot \lambda}\colon \Delta_x \to \nabla_x$ for all $x\in W_\mathrm{aff}^+$. For $m>0$, we set $\nabla_x / p^m \coloneqq \nabla_x \otimes ( \Zp / p^m \Zp )$ so that $[ \nabla_x / p^m ] = m \cdot [ \nabla_x / p ]$.

The following theorem gives a representation theoretic explanation for the (combinatorial) recursion formula from Theorem \ref{thm:recursion}.

\begin{Theorem} \label{thm:JSF}
	Let $x\in W_\mathrm{aff}^+$ and $s\in S$ such that $x<xs\in W_\mathrm{aff}^+$. Then there exist torsion $G_\Zp$-modules $A$, $B$, $C$, $D$ and $E$ and exact sequences
	\begin{align}
	0 & \longrightarrow A \longrightarrow \Theta_s \cok(c_x) \longrightarrow C \longrightarrow 0 , \label{eq:JSFseq1} \\ 
	0 & \longrightarrow B \longrightarrow \cok(c_{xs}) \longrightarrow C \longrightarrow 0 , \label{eq:JSFseq2} \\
	0 \longrightarrow D & \longrightarrow A \longrightarrow \cok(c_x) \longrightarrow E \longrightarrow 0 , \label{eq:JSFseq3} \\
	0 \longrightarrow D & \longrightarrow B \longrightarrow \nabla_x / p^m \longrightarrow E \longrightarrow 0 , \label{eq:JSFseq4}
	\end{align}
	where $m=\nu_p\big( m(xsx^{-1}) \cdot p \big)$. In particular, we have
	\[ [ \cok(c_{xs}) ] = [ \Theta_s \cok(c_x) ] - [ \cok(c_x) ] + [ \nabla_x / p^m ] = [ \cok(c_x) ] \cdot s + m \cdot [ \nabla_x / p ] \]
	in the Grothendieck group of torsion $G_\Zp$-modules.
\end{Theorem}
\begin{proof}
	Recall that we have fixed $\mu\in \overline{C}_\mathrm{fund}\cap X$ with $\Stab_{W_\mathrm{aff}}(\mu)=\{1,s\}$ and note that the unique reflection hyperplane containing $x\Cdot\mu$ is the one corresponding to $xsx^{-1}$. Using the Weyl dimension formula (and the fact that $p\geq h$), it is straightforward to see that
	\[ m = \nu_p\big( m(xsx^{-1}) \cdot p \big) = \nu_p\big( \dim \Delta(x\Cdot\mu) \big) . \]
	We adopt notation from Lemmas \ref{lem:Andersen1} and \ref{lem:Andersen2}, in particular $\beta= T_\mu^\lambda c_{x\Cdot\mu} = \Theta_s c_x$ and
	\begin{align*}
	i = \mathrm{adj}_1^{-1}(\id_{\Delta_\Zp(x\Cdot\mu)}) \colon & \Delta_{xs} \longrightarrow \Theta_s \Delta_{x} , \\
	r^\prime = \mathrm{adj}_2\big( \id_{\nabla_\Zp(x\Cdot\mu)} \big) \colon & \Theta_s \nabla_x \longrightarrow \nabla_x , \\
	\pi^\prime = \mathrm{adj}_2(\id_{\nabla_\Zp(x\Cdot\mu)}) \colon & \Theta_s \nabla_x \longrightarrow \nabla_{xs} .
	\end{align*}
	Furthermore, we denote by $q \colon \Theta_s \Delta_x \to \Theta_s \Delta_x / \ker(\pi^\prime\circ\beta)$ the canonical quotient homomorphism and by $\vartheta \colon \Theta_s \Delta_x / \ker(\pi^\prime\circ\beta) \to \nabla_{xs}$ the unique homomorphism such that $\pi^\prime\circ\beta = \vartheta \circ q$. Then there is a commutative diagram
	\begin{center}
		\begin{tikzpicture}
		\node (C1) at (0,0) {$0$};
		\node (C2) at (3,0) {$\Theta_s \Delta_x$};
		\node (C3) at (7,0) {$\Theta_s \nabla_x$};
		\node (C4) at (10,0) {$\cok(\beta)$};
		\node (C5) at (13,0) {$0$};
		\node (C6) at (0,-2) {$0$};
		\node (C7) at (3,-2) {$\Theta_s \Delta_x / \ker(\pi^\prime\circ\beta)$};
		\node (C8) at (7,-2) {$\nabla_{xs}$};
		\node (C9) at (10,-2) {$\cok(\vartheta)$};
		\node (C10) at (13,-2) {$0$};
		
		\draw[->] (C1) -- (C2);
		\draw[->] (C2) -- node[above] {\small $\beta$} (C3);
		\draw[->] (C3) -- (C4);
		\draw[->] (C4) -- (C5);
		\draw[->] (C6) -- (C7);
		\draw[->] (C7) -- node[below] {\small $\vartheta$} (C8);
		\draw[->] (C8) -- (C9);
		\draw[->] (C9) -- (C10);
		\draw[->] (C2) -- node[right] {\small $q$} (C7);
		\draw[->] (C3) -- node[right] {\small $\pi^\prime$} (C8);
		\draw[->] (C4) -- node[right] {\small $\pi^\prime_1$} (C9);
		\end{tikzpicture}
	\end{center}
	with exact rows, where $\pi^\prime_1$ is the homomorphism induced by $\pi^\prime$. As $q$ and $\pi^\prime$ are surjective, the snake lemma gives rise to an exact sequence
	\begin{equation} \label{eq:JSFseq5} 
	0 \to \ker(\pi^\prime\circ\beta) \to \ker(\pi^\prime) \to \ker(\pi^\prime_1) \to 0 \to 0 \to \cok(\overline{\pi^\prime}) \to 0 ,
	\end{equation}
	in particular $\cok(\pi^\prime_1)=0$ and we obtain a short exact sequence
	\begin{equation*}
	0 \to \ker(\pi^\prime_1) \to \cok(\beta) \to \cok(\vartheta) \to 0 .
	\end{equation*}
	By exactness of $\Theta_s$, we have $\cok(\beta)=\cok(\Theta_s c_x) \cong \Theta_s \cok(c_x)$ and with $A \coloneqq \ker(\pi^\prime_1)$ and $C \coloneqq \cok(\vartheta)$, we obtain the first exact sequence \eqref{eq:JSFseq1}.
	By the second commutative diagram in Lemma \ref{lem:Andersen2}, we have another commutative diagram
	\begin{center}
		\begin{tikzpicture}
		\node (C1) at (0,0) {$0$};
		\node (C2) at (3,0) {$\Delta_{xs}$};
		\node (C3) at (7,0) {$\nabla_{xs}$};
		\node (C4) at (10,0) {$\cok(c_{xs})$};
		\node (C5) at (13,0) {$0$};
		\node (C6) at (0,-2) {$0$};
		\node (C7) at (3,-2) {$\Theta_s \Delta_x / \ker(\pi^\prime\circ\beta)$};
		\node (C8) at (7,-2) {$\nabla_{xs}$};
		\node (C9) at (10,-2) {$C$};
		\node (C10) at (13,-2) {$0$};
		
		\draw[->] (C1) -- (C2);
		\draw[->] (C2) -- node[above] {\small $c_{xs}$} (C3);
		\draw[->] (C3) -- (C4);
		\draw[->] (C4) -- (C5);
		\draw[->] (C6) -- (C7);
		\draw[->] (C7) -- node[below] {\small $\vartheta$} (C8);
		\draw[->] (C8) -- (C9);
		\draw[->] (C9) -- (C10);
		\draw[->] (C2) -- node[right] {\small $q \circ i$} (C7);
		\draw[double equal sign distance] (C3) -- (C8);
		\draw[->] (C4) -- node[right] {\small $\eta$} (C9);
		\end{tikzpicture}
	\end{center}
	with exact rows, where $\eta$ is the induced homomorphism on the cokernels. This affords an exact sequence
	\[ 0 \to \ker(\eta) \to \cok(q\circ i) \to 0 \to \cok(\eta) \to 0 , \]
	so $\cok(\eta)=0$ and $\ker(\eta) \cong \cok(q \circ i)$. With $B \coloneqq \ker(\eta) \cong \cok(q \circ i)$, we obtain our second exact sequence
	\begin{equation*}
	0 \to B \to \cok(c_{xs}) \to C \to 0 .
	\end{equation*}
	Next note that we have a commutative diagram
	\begin{center}
		\begin{tikzpicture}
		\node (C1) at (0,0) {$0$};
		\node (C2) at (3,0) {$\Delta_{xs}$};
		\node (C3) at (7,0) {$\Theta_s \Delta_x$};
		\node (C4) at (11,0) {$\Delta_x$};
		\node (C5) at (14,0) {$0$};
		\node (C6) at (0,-2) {$0$};
		\node (C7) at (3,-2) {$\Delta_{xs}$};
		\node (C8) at (7,-2) {$\Theta_s \Delta_x / \ker(\pi^\prime\circ\beta)$};
		\node (C9) at (11,-2) {$\cok(q \circ i)$};
		\node (C10) at (14,-2) {$0$};
		
		\draw[->] (C1) -- (C2);
		\draw[->] (C2) -- node[above] {\small $i$} (C3);
		\draw[->] (C3) -- node[above] {\small $\pi$} (C4);
		\draw[->] (C4) -- (C5);
		\draw[->] (C6) -- (C7);
		\draw[->] (C7) -- node[below] {\small $q \circ i$} (C8);
		\draw[->] (C8) -- (C9);
		\draw[->] (C9) -- (C10);
		\draw[double equal sign distance] (C2) -- (C7);
		\draw[->] (C3) -- node[right] {\small $q$} (C8);
		\draw[->] (C4) -- node[right] {\small $q_1$} (C9);
		\end{tikzpicture}
	\end{center}
	with exact rows, where $q_1$ is induced by $q$. Therefore we have an exact sequence
	\[ 0 \to \ker(\pi^\prime \circ \beta) \to \ker(q_1) \to 0 \to 0 \to \cok(q_1) \to 0 , \]
	so $\cok(q_1)=0$ and $\ker(\pi^\prime \circ \beta) \cong \ker(q_1)$. Writing $\pi_0 \colon \ker(\pi^\prime\circ\beta) \to \Delta_\Zp(x\Cdot\lambda)$ for the restriction of $\pi$ to $\ker(\pi^\prime\circ\beta)$, we further have a short exact sequence
	\[ 0 \to \ker(\pi_0) \xrightarrow{\,\pi_0\,} \Delta_\Zp(xs\Cdot\lambda) \to \cok(q\circ i) \to 0 \]
	and it follows that $B \cong \cok(q\circ i) \cong \cok(\pi_0)$. Next, denote by $\beta_0 \colon \ker(\pi^\prime\circ\beta) \to \ker(\pi^\prime)$ the restriction of $\beta$ to $\ker(\pi^\prime\circ \beta)$ and note that the short exact sequence
	\[ 0 \to \ker(\pi^\prime\circ\beta) \xrightarrow{\beta_0} \ker(\pi^\prime) \to \ker(\pi^\prime_1) \to 0 \]
	from \eqref{eq:JSFseq5} yields an isomorphism $A=\ker(\pi^\prime_1) \cong \cok(\beta_0)$. Finally, let $r^\prime_0 \colon \ker(\pi^\prime) \to \nabla_x$ be the restriction of $r^\prime$ to $\ker(\pi^\prime)$. Using the first commutative diagram in Lemma \ref{lem:Andersen2}, we obtain a commutative diagram
	\begin{center}
		\begin{tikzpicture}
		\node (C1) at (0,0) {$0$};
		\node (C2) at (3,0) {$\ker(\pi^\prime\circ\beta)$};
		\node (C3) at (6,0) {$\ker(\pi^\prime)$};
		\node (C4) at (9,0) {$A$};
		\node (C5) at (12,0) {$0$};
		\node (C6) at (0,-2) {$0$};
		\node (C7) at (3,-2) {$\Delta_x$};
		\node (C8) at (6,-2) {$\nabla_x$};
		\node (C9) at (9,-2) {$\cok(c_x)$};
		\node (C10) at (12,-2) {$0$};
		
		\draw[->] (C1) -- (C2);
		\draw[->] (C2) -- node[above] {\small $\beta_0$} (C3);
		\draw[->] (C3) -- (C4);
		\draw[->] (C4) -- (C5);
		\draw[->] (C6) -- (C7);
		\draw[->] (C7) -- node[below] {\small $c_x$} (C8);
		\draw[->] (C8) -- (C9);
		\draw[->] (C9) -- (C10);
		\draw[->] (C2) -- node[right] {\small $\pi_0$} (C7);
		\draw[->] (C3) -- node[right] {\small $r^\prime_0$} (C8);
		\draw[->] (C4) -- node[right] {\small $\psi$} (C9);
		\end{tikzpicture}
	\end{center}
	with exact rows, where $\psi$ is induced by $r^\prime_0$. Now $\ker(\pi^\prime)=i^\prime(\nabla_{xs})$ and $r^\prime \circ i^\prime=p^m \cdot \id_{\nabla_x}$ by Lemma \ref{lem:Andersen1}, so $r^\prime_0$ is injective with
	\[ \cok(r^\prime_0) \cong \cok(r^\prime\circ i^\prime) \cong \nabla_x / p^m . \]
	As $B \cong \cok(\pi_0)$, the above diagram affords exact sequences
	\[ 0 \to \ker(\psi) \to B \to \nabla_x / p^m \to \cok(\psi) \to 0 \]
	and
	\[ 0 \to \ker(\psi) \to A \to \cok(c_x) \to \cok(\psi) \to 0 \]
	and with $D\coloneqq \ker(\psi)$ and $E\coloneqq \cok(\psi)$, we obtain the exact sequences \eqref{eq:JSFseq3} and \eqref{eq:JSFseq4}. The final claim easily follows from \eqref{eq:JSFseq1}--\eqref{eq:JSFseq4} as
	\begin{align*}
	[ \cok(c_{xs}) ] & = [ B ] + [ C ] \\
	& = [ B ] + [ \Theta_s \cok(c_x) ] - [ A ] \\
	& = [ B ] + [ \Theta_s \cok(c_x) ] - [ \cok(c_x) ] - [ D ] + [ E ] \\
	& = [ \Theta_s \cok(c_x) ] - [ \cok(c_x) ] + [ \nabla_x / p^m ] ,
	\end{align*}
	where $[ \nabla_x / p^m ] = m \cdot [ \nabla_x / p ]$ and $[\Theta_s \cok(c_x)] = [ \cok(c_x) ] \cdot (s+1)$.
\end{proof}

\begin{Corollary} \label{cor:upperbound}
	Let $\lambda\in C_\mathrm{fund}\cap X$, $x\in W_\mathrm{aff}^+$ and $s\in S$ such that $x<xs \in W_\mathrm{aff}^+$. Then
	\[ \max\big\{ i \geq 0 \mathrel{\big|} \Delta(xs\Cdot \lambda)^i \neq 0 \big\} \leq 2 \cdot \max\big\{ i \geq 0 \mathrel{\big|} \Delta(x\Cdot \lambda)^i \neq 0 \big\} + m , \]
	where $m=\nu_p\big( m(xsx^{-1}) \cdot p \big)$.
\end{Corollary}
\begin{proof}
	Recall from Lemma \ref{lem:DimensionSumFiltrationLength} that
	\[ \max\big\{ i \geq 0 \mathrel{\big|} \Delta(xs\Cdot \lambda)^i \neq 0 \big\} = \mathrm{tmax}\big( \cok(c_{xs}) \big) \]
	and
	\[ \max\big\{ i \geq 0 \mathrel{\big|} \Delta(x\Cdot \lambda)^i \neq 0 \big\} = \mathrm{tmax}\big( \cok(c_x) \big) . \]
	For a short exact sequence of torsion $\Zp$-modules $0 \to M^\prime \to M \to M^{\prime\prime} \to 0$, we have
	\[ \max\{ \mathrm{tmax}(M^\prime) , \mathrm{tmax}(M^{\prime\prime}) \} \leq \mathrm{tmax}(M) \leq \mathrm{tmax}(M^\prime) + \mathrm{tmax}(M^{\prime\prime}) , \]
	so \eqref{eq:JSFseq1}--\eqref{eq:JSFseq4} imply that
	\begin{multline*}
	\mathrm{tmax}\big( \cok(c_{xs}) \big) \leq  \mathrm{tmax}(B) + \mathrm{tmax}(C) \leq \mathrm{tmax}(D) + m + \mathrm{tmax}(C) \\ \leq \mathrm{tmax}(A) + m + \mathrm{tmax}(C) \leq 2 \cdot \mathrm{tmax}\big( \Theta_s \cok(c_x) \big) + m .
	\end{multline*}
	As $\Theta_s \cok(c_x)$ is a direct summand of a tensor product of $\cok(c_x)$ with some other $G_\Zp$-module, we have $\mathrm{tmax}\big( \Theta_s \cok(c_x) \big) \leq \mathrm{tmax}\big( \cok(c_x) \big)$ and it follows that
	\[ \mathrm{tmax}\big( \cok(c_{xs}) \big) \leq 2 \cdot \mathrm{tmax}\big( \cok(c_x) \big) + m , \]
	as required.
\end{proof}

\begin{Remark}
	If one assumes the Jantzen conjecture (stated as condition $(F,w,s)^+$ in Section II.C.9 of \cite{Jantzen}) then a result of H.H. Andersen shows that the composition multiplicities in the layers of the Jantzen filtration are given by coefficients of inverse parabolic Kazhdan-Lusztig polynomials, see Theorem 2.4 in \cite{AndersenJantzenFiltration}. In this case, it turns out that the factor $2$ in the upper bound in Corollary \ref{cor:upperbound} is superfluous. Note that the Jantzen conjecture would imply the Lusztig character formula, so this can only be expected to happen in the so-called \emph{Jantzen region} and when $p$ is very large.
\end{Remark}

\begin{Remark} \label{rem:AndersenBound}
	The following alternative way of obtaining an upper bound on the length of the Jantzen filtration has been pointed out to the author by H.H. Andersen: For $\nu \in X^+$, one can mimic the construction of the homomorphism $c_\nu \colon \Delta_\Zp(\nu) \to \nabla_\Zp(\nu)$ from Section II.8 in \cite{Jantzen} to define a homomorphism $c_\nu^\prime \colon \nabla_\Zp(\lambda) \to \Delta_\Zp(\lambda)$ such that
	\[ c_\nu \circ c_\nu^\prime = \Big( \prod_{\beta\in\Phi^+} \big( \langle \nu+\rho , \beta^\vee \rangle - 1 \big) ! \Big) \cdot \id_{\nabla_\Zp(\nu)} . \]
	Note that J.C. Jantzen works over the integers rather than the $p$-adic integers. The homomorphism corresponding to $c_\nu$ is called $\tilde{T}_{w_0}(w_0\Cdot \nu)$ and defined in Section II.8.16 in \cite{Jantzen}. In order to define the homomorphism $c_\nu^\prime$, one replaces the homomorphism from II.8.15(3) in \cite{Jantzen} by the homomorphism from II.8.15(4); the identity for the composition of $c_\nu$ with $c_\nu^\prime$ follows from the remark after II.8.15(4).
	Now there is a canonical surjective homomorphism from $\cok(c_\nu \circ c_\nu^\prime)$ to $\cok(c_\nu)$ and it follows that
	\begin{multline*}
		\hspace{.3cm} \max\{ i \mid \Delta(\nu)^i \neq 0 \} = \mathrm{tmax}\big( \cok(c_\nu) \big) \leq \mathrm{tmax}\big( \cok(c_\nu \circ c_\nu^\prime) \big) \\ = \nu_p\Big( \prod_{\beta\in\Phi^+} \big( \langle \nu+\rho , \beta^\vee \rangle - 1 \big) ! \Big) = \sum_{\beta\in\Phi^+} \sum_{i=1}^{\langle \nu+\rho , \beta^\vee \rangle - 1} \nu_p(i) . \hspace{.3cm}
	\end{multline*}
	If $\nu$ is $p$-regular, say $\nu \in x \Cdot C_\mathrm{fund}$ for some $x\in W_\mathrm{aff}^+$, then arguing as in the second paragraph of Section \ref{sec:rewriting}, we get
	\[ \max\{ i \mid \Delta(\nu)^i \neq 0 \} \leq \sum_{\beta\in\Phi^+} \sum_{i=1}^{\langle \nu+\rho , \beta^\vee \rangle} \nu_p(i) = \sum_{ s \in R_L(x) } \nu_p( m(s) \cdot p ) . \]
	If we further assume that $\langle \nu+\rho , \alpha_\mathrm{h}^\vee \rangle<p^2$ (so that $x\Cdot C_\mathrm{fund}$ lies in the lowest $p^2$-alcove) then $m(s)<p$ for all $s\in R_L(x)$ and it follows that
	$ \max\{ i \mid \Delta(\nu)^i \neq 0 \} \leq \abs{R_L(x)} = \ell(x) $.
\end{Remark}

Neither of the bounds on the length of the Jantzen filtration from Corollary \ref{cor:upperbound} and Remark \ref{rem:AndersenBound} is sharp, except in special cases. Note that Corollary \ref{cor:upperbound} can also be used to obtain an explicit bound on the length of the Jantzen filtration of a Weyl module $\Delta(\nu)$ with $\nu\in x \Cdot C_\mathrm{fund}$ for some $x\in W_\mathrm{aff}^+$ in terms of $\ell(x)$, but this bound is much bigger than the one that was given in Remark \ref{rem:AndersenBound}: By induction on $\ell(x)$, one obtains
\[ \max\{ i \mid \Delta(\nu)^i \neq 0 \} \leq 2^{\ell(x)} - 1 \]
when $x\Cdot C_\mathrm{fund}$ lies in the lowest $p^2$-alcove. Corollary \ref{cor:upperbound} becomes useful in situations where the length of the Jantzen filtration of a Weyl module (with $p$-regular highest weight) is already known and one wants to bound the length of the Jantzen filtration of a Weyl module with highest weight in an adjacent alcove.

\subsection{Euler characteristics} \label{subsec:Eulercharacteristic}

For finitely generated $G_{\Z_p}$-modules $M$ and $N$, the $\Z_p$-module $\Hom_{G_\Zp}(M,N)$ decomposes naturally as the direct sum of its torsion submodule $\Hom_{G_\Zp}(M,N)_\mathrm{tor}$ and its unique maximal (torsion-)free submodule $\Hom_{G_\Zp}(M,N)_\mathrm{fr}$. For all $i>0$, the $\Ext$-group $\Ext_{G_\Zp}^i(M,N)$ is a torsion $\Z_p$-module and following U. Kulkarni, we define the \emph{torsion Euler characteristic} by
\[ E(M,N) = \ell\big( \Hom_{G_\Zp}(M,N)_\mathrm{tor} \big) + \sum_{i>0} (-1)^i \cdot \ell\big( \Ext_{G_\Zp}(M,N) \big) . \]
Then U. Kulkarni shows in Section 1.4 of \cite{KulkarniJSF} that
$ \JSF_\mu = - \sum_{\lambda\in X^+} E\big( \Delta_{\Z_p}(\lambda) , \Delta_{\Z_p}(\mu) \big) \cdot \chi_\lambda $
for all $\mu\in X^+$,
which amounts to \emph{Kulkarni's formula}
\begin{equation} \label{eq:Kulkarni}
	\langle \chi_\lambda^* , \JSF_\mu \rangle = - E\big( \Delta_{\Z_p}(\lambda) , \Delta_{\Z_p}(\mu) \big) .
\end{equation}
For a tilting module $T$ with good filtration multiplicities $a_\mu \coloneqq \dim \Hom_G(\Delta(\mu),T)$ for $\mu\in X^+$, it is shown in the proof of Theorem 5.2 in \cite{AndersenKulkarni} that
\[ \langle \ASF_\lambda , \ch T \rangle = - \sum_{\mu\in X^+} a_\mu \cdot E\big( \Delta_{\Z_p}(\lambda) , \Delta_{\Z_p}(\mu) \big) , \]
so $\langle \ASF_\lambda , \chi_\mu \rangle = - E\big( \Delta_{\Z_p}(\lambda) , \Delta_{\Z_p}(\mu) \big) = \langle \chi_\lambda^* , \JSF_\mu \rangle$. (This is the representation theoretic explanation of the duality formula mentioned in Remark \ref{rem:singularduality}.)
Using the torsion Euler-characteristic, we can give yet another explanation for the recursion formula
\[ \JSF_{xs} = \nu_p( m(xsx^{-1}) \cdot p ) \cdot N_x + \JSF_x \cdot s , \]
via the following Lemma. (We return to the notational conventions set up before Theorem \ref{thm:JSF}.)

\begin{Lemma} \label{lem:recursionEulercharacteristic}
	Let $x,y\in W_\mathrm{aff}^+$ and $s\in S$ such that $x<xs\in W_\mathrm{aff}^+$. Then
	\[ E(\Delta_y,\Delta_{xs}) = E(\Delta_y,\Theta_s \Delta_x) - E(\Delta_y,\Delta_x) - \delta_{x,y} \cdot \nu_p\big( m(xsx^{-1}) \cdot p \big) \]
	and
	\[ E(\Delta_y,\Theta_s\Delta_x) = E(\Theta_s\Delta_y,\Delta_x) = \begin{cases}
	E(\Delta_y,\Delta_x) + E(\Delta_{ys},\Delta_x) & \text{if } ys\in W_\mathrm{aff}^+ , \\
	0 & \text{otherwise} .
	\end{cases} \]
\end{Lemma}
\begin{proof}
	First, the short exact sequence
	\[ 0 \longrightarrow \Delta_{xs} \longrightarrow \Theta_s \Delta_x \longrightarrow \Delta_x \longrightarrow 0  \]
	gives rise to a long exact sequence
	\begin{multline} \label{eq:longexact1}
	0 \to \Hom_{G_\Zp}(\Delta_y,\Delta_{xs}) \to \Hom_{G_\Zp}(\Delta_y,\Theta_s \Delta_x) \to \Hom_{G_\Zp}(\Delta_y,\Delta_x) \\ \to \Ext_{G_\Zp}^1(\Delta_y,\Delta_{xs}) \to \Ext_{G_\Zp}^1(\Delta_y,\Theta_s \Delta_x) \to \Ext_{G_\Zp}^1(\Delta_y,\Delta_x) \to \cdots
	\end{multline}
	where the three initial terms are zero unless $y\in\{x,xs\}$ (see Section 1.2 in \cite{AndersenFiltrationsTilting}).
	If $y=xs$ then the third term in \eqref{eq:longexact1} is zero and the first term is isomorphic to the second, so the first claim follows for all $y\neq x$. For $y=x$, we have
	\[ \Ext_{G_\Zp}^1( \Delta_x , \Theta_s \Delta_x ) \cong \Ext_{G_\Zp}^1( T_\lambda^\mu \Delta_x , T_\lambda^\mu \Delta_x ) \cong 0 \]
	and the $\Zp$-linear map from $\Hom_{G_\Zp}(\Delta_x,\Theta_s \Delta_x)\cong \Zp$ to $\Hom_{G_\Zp}(\Delta_x,\Delta_x)\cong \Zp$ identifies with multiplication by $p^m$ for $m=\nu_p\big( m(xsx^{-1}) \cdot p \big)$ by Lemma \ref{lem:Andersen1}. Thus the beginning of the long exact sequence \eqref{eq:longexact1} reduces to
	\[0 \to 0 \to \Zp \to \Zp \to \Zp / p^m \Zp \to 0 \to \cdots , \]
	and we conclude that
	\[ E(\Delta_x,\Delta_{xs}) = E(\Delta_x,\Theta_s\Delta_x) - E(\Delta_x,\Delta_x) - m . \]
	The first equality in the second claim follows from the adjunction between $T_\lambda^\mu$ and $T_\mu^\lambda$. If $ys\notin W_\mathrm{aff}^+$ then $\Theta_s \Delta_y=0$ and therefore
	$E(\Theta_s\Delta_y,\Delta_x) = 0$
	as claimed, so now suppose that $ys\in W_\mathrm{aff}^+$. In order to show the second equality, we may assume without loss of generality that $y<ys$ since $\Theta_s \Delta_y \cong \Theta_s \Delta_{ys}$.
	Then the short exact sequence
	\[ 0 \longrightarrow \Delta_{ys} \longrightarrow \Theta_s \Delta_y \longrightarrow \Delta_y \longrightarrow 0 \]
	gives rise to a long exact sequence
	\begin{multline} \label{eq:longexact2}
	0 \to \Hom_{G_\Zp}(\Delta_y,\Delta_x) \to \Hom_{G_\Zp}(\Theta_s\Delta_y,\Delta_x) \to \Hom_{G_\Zp}(\Delta_{ys},\Delta_x) \\ \to \Ext_{G_\Zp}^1(\Delta_y,\Delta_x) \to \Ext_{G_\Zp}^1(\Theta_s \Delta_y,\Delta_x) \to \Ext_{G_\Zp}^1(\Delta_{ys},\Delta_x) \to \cdots .
	\end{multline}
	Again, the three initial terms are zero unless $x\in\{ y , ys \}$ and the condition $y<ys$ implies that $x\neq ys$. If $x=y$ then the third term in \eqref{eq:longexact2} is zero and the first term is isomorphic to the second, so the second claim follows in all cases.
\end{proof}

\begin{Remark} \label{rem:recursionEulercharacteristic}
	Using Lemma \ref{lem:recursionEulercharacteristic} and Kulkarni's formula $\langle N_y^* , \JSF_x \rangle = - E( \Delta_y , \Delta_x )$ from \eqref{eq:Kulkarni}, we can recover the recursion formula as follows: Let $x,y\in W_\mathrm{aff}^+$ and $s\in S$ such that $x<xs \in W_\mathrm{aff}^+$. If $ys\in W_\mathrm{aff}^+$ then
	\begin{align*}
	\big\langle N_y^* , \JSF_{xs} \big\rangle & = - E( \Delta_y , \Delta_{xs} ) \\
	& = - E(\Delta_y,\Theta_s \Delta_x) + E(\Delta_y,\Delta_x) + \delta_{x,y} \cdot \nu_p\big( m(xsx^{-1}) \cdot p \big) \\
	& = - E(\Delta_{ys},\Delta_x) + \delta_{x,y} \cdot \nu_p\big( m(xsx^{-1}) \cdot p \big) \\
	& = \big\langle N_y^* \cdot s , \JSF_x \big\rangle + \big\langle N_y^* , \nu_p\big( m(xsx^{-1}) \cdot p \big) \cdot N_x \big\rangle \\
	&= \big\langle N_y^* , \JSF_x \cdot s + \nu_p\big( m(xsx^{-1}) \cdot p \big) \cdot N_x \big\rangle 
	\end{align*}
	and if $ys\notin W_\mathrm{aff}^+$ then
	\begin{align*}
	\big\langle N_y^* , \JSF_{xs} \big\rangle & = - E( \Delta_y , \Delta_{xs} ) \\
	& = - E(\Delta_y,\Theta_s \Delta_x) + E(\Delta_y,\Delta_x) + \delta_{x,y} \cdot \nu_p\big( m(xsx^{-1}) \cdot p \big) \\
	& = E(\Delta_y,\Delta_x) + \delta_{x,y} \cdot \nu_p\big( m(xsx^{-1}) \cdot p \big) \\
	& = \big\langle N_y^* \cdot s , \JSF_x \big\rangle + \big\langle N_y^* , \nu_p\big( m(xsx^{-1}) \cdot p \big) \cdot N_x \big\rangle \\
	& = \big\langle N_y^* , \JSF_x \cdot s + \nu_p\big( m(xsx^{-1}) \cdot p \big) \cdot N_x \big\rangle 
	\end{align*}
	since $N_y^* \cdot s= - N_y^*$. It follows that
	\[ \JSF_{xs} = \JSF_x \cdot s + \nu_p\big( m(xsx^{-1}) \cdot p \big) \cdot N_x , \]
	as claimed.
\end{Remark}

\section{An example} \label{sec:anexample}

Suppose that $G$ is of type $\mathrm{A}_n$ for some $n \geq 2$ and that $p \geq h = n+1$. Fix a labeling $\Delta=\{ \alpha_1 , \ldots , \alpha_n \}$ of the simple roots such that
\[ \langle \alpha_i , \alpha_j^\vee \rangle = \begin{cases} 2 & \text{if } i=j , \\ -1 & \text{if } \abs{i-j}=1 , \\ 0 & \text{otherwise} \end{cases} \]
and write $S=\{ s_0 , s_1 , \ldots, s_n \}$ with $s_0 = s_{\alpha_\mathrm{h},1}$ and $s_i = s_{\alpha_i}$ for $i=1,\ldots,n$. A product of two simple reflections $s_i$ and $s_j$ with $i < j$ has order $3$ if $j=i+1$ or $(i,j)=(0,n)$ and it has order $2$ otherwise; see Section 4.7 in \cite{HumphreysCoxeter}. We write $\omega_1,\ldots,\omega_n\in X^+$ for the fundamental dominant weights with respect to $\Delta$, that is $\langle \omega_i , \alpha_j^\vee \rangle = \delta_{ij}$, and fix the weight $\lambda=(p-n-1) \cdot \omega_n \in C_\mathrm{fund}$. 
We consider the elements $x_0 \coloneqq e$ and $x_i \coloneqq s_0 s_1 \cdots s_{i-1} \in W_\mathrm{aff}$ for $1\leq i \leq n$. By induction on $i$, one sees that
\[ s_1 \cdots s_{i-1} \Cdot \lambda = \lambda - \alpha_{i-1} - 2 \alpha_{i-2} - \cdots - (i-1) \alpha_1 = \lambda + \omega_i - i \cdot \omega_1 \eqqcolon \lambda_i \]
for $1 \leq i \leq n$. As $\alpha_\mathrm{h}= \alpha_1 + \cdots + \alpha_n = \omega_1 + \omega_n$ and $\alpha_\mathrm{h}^\vee = \alpha_1^\vee + \cdots + \alpha_n^\vee$, we further obtain
\[ \langle \lambda_i + \rho , \alpha_\mathrm{h}^\vee \rangle = \big\langle (p-n-1) \cdot \omega_n + \omega_i - i \cdot \omega_1 + \rho , \alpha_\mathrm{h}^\vee \big\rangle = p - i \]
and
\[ x_i \Cdot \lambda = s_0 \Cdot \lambda_i 
 = \lambda_i - \big( \langle \lambda_i + \rho , \alpha_\mathrm{h}^\vee \rangle - p \big) \cdot \alpha_\mathrm{h} 
 =  \lambda_i + i \cdot \alpha_\mathrm{h} 
 = \omega_i + (p-n-1+i) \cdot \omega_n
\]
for $1 \leq i \leq n$.
Note that $x_i \Cdot \lambda - x_{i-1} \Cdot \lambda = \alpha_i + \cdots + \alpha_n \eqqcolon \beta_i$ for $1 \leq i \leq n$. In fact, it is straightforward to see that $x_i\Cdot\lambda = s_{\beta_i,1} x_{i-1} \Cdot \lambda$ and therefore $x_{i-1} s_{i-1} x_{i-1}^{-1} = x_i x_{i-1}^{-1} = s_{\beta_i,1}$.
Also note that $x_n \Cdot \lambda = p \cdot \omega_n$ and therefore $x_n \Cdot C_\mathrm{fund} = C_\mathrm{fund} + p \cdot \omega_n$.

We can now compute $\JSF_{x_i}$ for $0 \leq i \leq n$ using the recursion formula and induction on $i$.

\begin{Lemma} \label{lem:JSFexample}
	For $0 \leq i \leq n$, we have $\JSF_{x_i} = \sum_{j=1}^{i} (-1)^{j+1} \cdot N_{x_{i-j}}$.
\end{Lemma}
\begin{proof}
	For $i=0$, we have $x_0 = e$ and $\Delta(x_0 \Cdot \mu) = \Delta(\mu) = L(\mu)$ for all $\mu \in C_\mathrm{fund} \cap X$ by the linkage principle, so $\JSF_{x_0} = 0$ as required. For $i=1$, the recursion formula from Theorem \ref{thm:recursion} yields
	\[ \JSF_{x_1} = \JSF_{s_0} = N_{s_0} + \JSF_e \cdot s_0 = N_{s_0}. \]
	Now suppose that $1<i\leq n$ and that the claim is true for $x_{i-1}$. We have $x_i = x_{i-1} s_{i-1} > x_{i-1}$ and $\nu_p\big( m( x_{i-1} s_{i-1} x_{i-1}^{-1} ) \cdot p \big) = 1$ because $x_{i-1} s_{i-1} x_{i-1}^{-1} = s_{\beta_i,1}$. By the induction hypothesis and the recursion formula, we have
	\[ \JSF_{x_i} = N_{x_{i-1}} + \JSF_{x_{i-1}} \cdot s_{i-1} = N_{x_{i-1}} + \sum_{j=1}^{i-1} (-1)^{j+1} \cdot N_{x_{i-1-j}} \cdot s_{i-1} . \]
	Now $s_{i-1}$ commutes with $s_{k-1}$ for all $k<i-1$, hence
	\[ N_{x_k} \cdot s_{i-1} = N_{x_k s_{i-1}} = N_{s_{i-1} x_k} = - N_{x_k} \]
	for $k < i-1$ and it follows that
	\[ \JSF_{x_i} = N_{x_{i-1}} + \sum_{j=1}^{i-1} (-1)^{j+1} \cdot N_{x_{i-1-j}} \cdot s_{i-1} = N_{x_{i-1}} - \sum_{j=1}^{i-1} (-1)^{j+1} \cdot N_{x_{i-1-j}} = \sum_{j=1}^{i} (-1)^{j+1} \cdot N_{x_{i-j}} , \]
	as claimed.
\end{proof}

Recall from Section \ref{sec:rewriting} that for every $\mu \in C_\mathrm{fund} \cap X$, there is an isomorphism between $M_\mathrm{asph}$ and the character lattice of $\Rep_\mu(G)$ that sends $N_x$ to $\chi(x\Cdot\mu)$ and $\JSF_x$ to $\JSF_{x\Cdot\mu}$ for all $x\in W_\mathrm{aff}^+$. Thus the equation in Lemma \ref{lem:JSFexample} can be rewritten as
\[ \sum_{j>0} \ch \Delta(x_i\Cdot\mu)^j = \JSF_{x_i \Cdot \mu} = \sum_{j=1}^{i} (-1)^{j+1} \cdot \chi(x_{i-j}\Cdot\mu) = \sum_{j=1}^{i} (-1)^{j+1} \cdot \ch \Delta(x_{i-j}\Cdot\mu) . \]
Using this observation, we can determine the composition series of the Weyl modules with highest weight in the alcoves $x_i \Cdot C_\mathrm{fund}$ for $1 \leq i \leq n$.

\begin{Lemma} \label{lem:examplecompositionseries}
	Let $\mu \in C_\mathrm{fund} \cap X$ and $1 \leq i \leq n$. In the Grothendieck group of $\Rep(G)$, we have
	\[ [ \Delta(x_i \Cdot\mu) ] = [ L(x_i\Cdot\mu) ] + [ L(x_{i-1}\Cdot\mu) ] \qquad \text{and} \qquad [ L(x_i\Cdot\mu) ] = \sum_{j=0}^i (-1)^j \cdot [ \Delta(x_{i-j}\Cdot\mu) ] . \]
\end{Lemma}
\begin{proof}
	We prove the claim by induction on $i$. For $i=1$, we have
	\[ \sum_{j>0} \ch \Delta(x_1 \Cdot \mu)^j = \JSF_{x_1 \Cdot \mu} = \ch \Delta(\mu) = \ch L(\mu) \]
	because $\Delta(\mu)=L(\mu)$. This implies that $\Delta(x_1 \Cdot \mu)^1 \cong L(\mu)$ and therefore
	\[ [ \Delta(x_1 \Cdot \mu) ] = [ L(x_1\Cdot\mu) ] + [ L(x_0\Cdot\mu) ] \qquad \text{and} \qquad [ L(x_1 \Cdot \mu) ] = [ \Delta(x_1\Cdot\mu) ] - [ \Delta(x_0\Cdot\mu) ] , \]
	as claimed. Now suppose that $i>1$ and that $[ L(x_{i-1}\Cdot\mu) ] = \sum_{j=1}^i (-1)^{j+1} \cdot [ \Delta(x_{i-j}\Cdot\mu) ]$.
	Then we have
	\[ \sum_{j>0} \ch \Delta(x_i \Cdot \mu)^j = \JSF_{x_i \Cdot \mu} = \sum_{j=1}^{i} (-1)^{j+1} \cdot \ch \Delta(x_{i-j}\Cdot\mu) = \ch L(x_{i-1}\Cdot\mu) \]
	and it follows that $\Delta(x_i\Cdot\mu)^1 \cong L(x_{i-1} \Cdot \mu)$. As before, we conclude that
	\[ [ \Delta(x_i \Cdot\mu) ] = [ L(x_i\Cdot\mu) ] + [ L(x_{i-1}\Cdot\mu) ] \]
	and
	\[ [ L(x_i\Cdot\mu) ] = [ \Delta(x_i\Cdot\mu) ] - [ L(x_{i-1}\Cdot\mu) ] = \sum_{j=0}^i (-1)^j \cdot [ \Delta(x_{i-j}\Cdot\mu) ] , \]
	as required.
\end{proof}

\begin{Remark}
	Lemmas \ref{lem:JSFexample} and \ref{lem:examplecompositionseries} are actually special cases of a result of J.C.\ Jantzen, see the second paragraph of Section 7 in \cite{JantzenKontravarianteFormen}.
	For an element $x\in W_\mathrm{aff}^+$ with reduced expression $x = t_1 \cdots t_m$ and such that $x\Cdot C_\mathrm{fund}$ lies in the lowest $p^2$-alcove and the set $\{ y \in W_\mathrm{aff}^+ \mid y<x \}$ is linearly ordered with respect to the Bruhat order, J.C.\ Jantzen shows (for all types of root systems) that the Jantzen sum formula is given by
	\[ \JSF_x = \sum_{i=0}^{m-1} (-1)^{m-i-1} \cdot N_{t_1 \cdots t_i} . \]
	As in Lemma \ref{lem:examplecompositionseries}, one concludes that a Weyl module with highest weight in the alcove $x\Cdot C_\mathrm{fund}$ has precisely two composition factors, with highest weights in $x\Cdot C_\mathrm{fund}$ and $t_1 \cdots t_{m-1} \Cdot C_\mathrm{fund}$, respectively.
\end{Remark}

Let us conclude by illustrating the two upper bounds on the length of the Jantzen filtration from Corollary \ref{cor:upperbound} and Remark \ref{rem:AndersenBound}.

\begin{Remark}
	From the proof of Lemma \ref{lem:examplecompositionseries}, we see that $\Delta(x_i \Cdot \mu)^1 \cong L(x_{i-1} \Cdot \mu)$ and $\Delta(x_i \Cdot \mu)^2 = 0$ for all $1 \leq i \leq n$ and $\mu \in C_\mathrm{fund} \cap X$, so
	$ \max\big\{ j \mathrel{\big|} \Delta(x_i \Cdot \mu)^j \neq 0 \big\} = 1 $. We have $x_i < x_i s_n \in W_\mathrm{aff}^+$ and Corollary \ref{cor:upperbound} yields an upper bound on the length of the Jantzen filtration of $\Delta(x_i s_n \Cdot \mu)$:
	\[ \max\big\{ j \mathrel{\big|} \Delta(x_i s_n \Cdot \mu)^j \neq 0 \big\} \leq 2 \cdot 1 + 1 = 3 . \]
	For the same Weyl module, the bound from Remark \ref{rem:AndersenBound} yields
	\[ \max\big\{ j \mathrel{\big|} \Delta(x_i s_n \Cdot \mu)^j \neq 0 \big\} \leq \ell(x_i s_n) = i+1 . \]
	We claim that $\max\big\{ j \mathrel{\big|} \Delta(x_1 s_n \Cdot \mu)^j \neq 0 \big\} = 1$ and $\max\big\{ j \mathrel{\big|} \Delta(x_i s_n \Cdot \mu)^j \neq 0 \big\} = 2$ for $1 < i \leq n$. Indeed, by the recursion formula and Lemma \ref{lem:JSFexample}, we have
	\[ \JSF_{x_i s_n} = N_{x_i} + \JSF_{x_i} \cdot s_n = N_{x_i} + \sum_{j=1}^{i-1} (-1)^{j+1} \cdot N_{x_{i-j} s_n} + (-1)^i \cdot N_e \]
	and as in Lemma \ref{lem:examplecompositionseries}, it follows that $\Delta(x_1 s_n \Cdot \mu)^1 \cong L(x_1 \Cdot \mu)$ and $\Delta(x_1 s_n \Cdot \mu)^2 = 0$.
	Using Lemma \ref{lem:examplecompositionseries} and induction on $i$, one sees that
	\begin{equation} \label{eq:JSFexample1}
		\sum_{j>0} \ch \Delta(x_i s_n \Cdot \mu)^j = \ch L(x_i \Cdot \mu) + \ch L(x_{i-1} s_n \Cdot \mu) + 2 \cdot L(x_{i-1} \Cdot \mu) + \delta_{i,2} \cdot \ch L(\mu)
	\end{equation}
	and
	\begin{equation} \label{eq:JSFexample2}
	[ \Delta(x_i s_n \Cdot \mu) ] = [ L(x_i s_n \Cdot \mu) ] + [ L(x_i \Cdot \mu) ] + [ L(x_{i-1} s_n \Cdot \mu) ] + [ L(x_{i-1} s_n \Cdot \mu) ] + \delta_{i,2} \cdot [L(\mu)]
	\end{equation}
	for $1<i\leq n$, where $\delta_{i,2}=0$ if $i\neq 2$ and $\delta_{2,2}=1$. Note that \eqref{eq:JSFexample1} directly implies that the simple modules $L(x_i \Cdot \mu)$ and $L(x_{i-1} s_n \Cdot \mu)$ appear with multiplicity one in a composition series of the Weyl module $\Delta(x_i s_n \Cdot \mu)$ (and so  does $L(\mu)$ if $i=2$). The multiplicity of $L(x_{i-1} \Cdot \mu)$ in $\Delta(x_i s_n \Cdot \mu)$ equals the multiplicity of $L(x_{i-1} \Cdot \mu)$ in $\Delta(x_i \Cdot \mu)$ by Proposition 7.18 in \cite{Jantzen}, and the latter multiplicity is also equal to one by Lemma \ref{lem:examplecompositionseries}. Combining \eqref{eq:JSFexample1} and \eqref{eq:JSFexample2}, we see that $\Delta(x_i s_n \Cdot \mu)^2 \cong L(x_{i-1} \Cdot \mu)$ and $\Delta(x_i s_n \Cdot \mu)^3=0$, so $\max\big\{ j \mathrel{\big|} \Delta(x_i s_n \Cdot \mu)^j \neq 0 \big\} = 2$ as claimed.
\end{Remark}

\bibliographystyle{alpha}
\bibliography{JSF}
\bigskip

{ \footnotesize
\textsc{\'Ecole Polytechnique Federale de Lausanne, 1015 Lausanne, Switzerland}

\textit{E-mail address:} \texttt{jonathan.gruber@epfl.ch}
}

\end{document}